\documentclass[12pt]{article}
\usepackage{graphicx} % Required for inserting images
\usepackage{amsmath}
\usepackage{amssymb}
\usepackage[english]{babel}
\usepackage{amsthm}
\usepackage{mathtools}
\usepackage{xcolor}
\usepackage{geometry}
\usepackage{setspace}
\usepackage[colorlinks=true, linkcolor=blue, citecolor=blue, urlcolor=blue]{hyperref}
\usepackage{comment}
\usepackage{enumitem}
\usepackage{subcaption}
\usepackage{caption}
\usepackage{authblk}
\usepackage{csquotes}
\usepackage{academicons}
\usepackage{xcolor}

\usepackage{color,soul}
\setstcolor{red}
\sethlcolor{yellow}

\title{Nonsmooth bifurcations in families of one-dimensional piecewise-linear quasiperiodically forced maps}

\author{ Rafael Martinez-Vergara and Joan Carles Tatjer\\
{\small Departament de Matem\`atiques i Inform\`atica\\
Universitat de Barcelona\\
Gran Via 585, 08007 Barcelona, Spain.}
}

\date{\small Date: \today.}

\newcommand{\N}{\mathbb{N}}
\newcommand{\Z}{\mathbb{Z}}
\newcommand{\T}{\mathbb{T}}
\newcommand{\R}{\mathbb{R}}

\newcommand{\Q}{\mathbb{Q}}

\theoremstyle{plain}
\newtheorem{defn}{Definition}[section]
\newtheorem{prop}[defn]{Proposition}
\newtheorem{thm}[defn]{Theorem}
\newtheorem{lemma}[defn]{Lemma}
\newtheorem{corollary}[defn]{Corollary}

\theoremstyle{remark}
\newtheorem{remark}[defn]{Remark}
\newtheorem{assumption}[defn]{Assumption}

\DeclarePairedDelimiter\abs{\lvert}{\rvert}
\DeclarePairedDelimiter\norm{\lVert}{\rVert}

\theoremstyle{definition}
\newtheorem{definition}{Definition}[section]

\begin{document}

\maketitle

\begin{abstract}
    We study nonsmooth bifurcations of four types of families of one-dimensional quasiperiodically forced maps of the form $F_i(x,\theta) = (f_i(x,\theta), \theta+\omega)$ for $i=1,\dots,4$, where $x$ is real, $\theta\in\T$ is an angle, $\omega$ is an irrational frequency, and $f_i(x,\theta)$ is a real piecewise linear map with respect to $x$. The first two types of families $f_i$ have a symmetry with respect to $x$, and the other two could be viewed as quasiperiodically forced piecewise-linear versions of saddle-node and period-doubling bifurcations. The four types of families depend on two real parameters, $a\in\R$ and $b\in\R$. Under certain assumptions for $a$, we prove the existence of a continuous map $b^*(a)$ where for $b=b^*(a)$ there exists a nonsmooth bifurcation for these types of systems. In particular we prove that for $b=b^*(a)$ we have a strange nonchaotic attractor.
    It is worth to mention that the four families are piecewise-linear versions of smooth families 
    which seem to have nonsmooth bifurcations. Moreover, as far as we know, we give the first example of a family with a nonsmooth period-doubling bifurcation.


\end{abstract}

\newpage
\section{Introduction}
In this paper, we investigate the dynamics of four families of piecewise-linear quasiperiodically forced maps. For these families, we want to study a certain type of (nonsmooth) bifurcations involving invariant and two-periodic curves. These bifurcations are related to the existence of Strange Nonchaotic Attractors (SNA). A two-periodic curve is a curve which is invariant under the second iterate of the system, see Definition \ref{def_twi_inv} for a formal definition. For the cases we consider, we say that a bifurcation is nonsmooth whenever, at the bifurcation parameter, some of the invariant or two-periodic curves involved in the bifurcation fail to be continuous. We prove that the families we study exhibit nonsmooth bifurcations.

The motivation to study these four maps is due to the previous study in \cite{jorba2024nonsmooth}. In this work, the authors study the following quasiperiodically forced map, for $a>0$, $b\in\R$ and $\omega\not\in2\pi\mathbb{Q}$,
\begin{align*}
    \begin{cases}
        \Bar{x} = h_{a}(x) - b \sin(\theta),\\
        \Bar{\theta} = \theta + \omega \mod2\pi,
    \end{cases}
    &\hspace{1cm}
    h_{a}(x)= 
    \begin{cases}
        -\frac{\pi}{2} &\text{if }x\in(-\infty,-\frac{\pi}{2a}),\\
        ax & \text{if }x\in[-\frac{\pi}{2a},\frac{\pi}{2a}],\\
        \frac{\pi}{2} &\text{if }x\in(\frac{\pi}{2a},\infty).
    \end{cases}
\end{align*}
They prove the existence of a value of a parameter $b=b^*(a)$ where a nonsmooth pitchfork bifurcation occurs. 
This system is a piecewise-linear approximation of the smooth system studied in \cite{Jorba2018Mu} and \cite{jager2009}.
The present study generalizes and extends the results of \cite{jorba2024nonsmooth} by identifying further models for which similar statements hold and establishing additional properties of the systems. We will consider four quasiperiodically forced maps of the form
\begin{align*}
    \begin{cases}
        \Bar{x} = h(x) + b g(\theta),\\
        \Bar{\theta} = \theta + \omega \mod2\pi,
    \end{cases}
\end{align*}
where $b\in\R$ is a real parameter, $\omega\not\in2\pi\mathbb{Q}$, $h\colon\R\to\R$ is a piecewise-linear continuous function and $g\colon\T\to\R$ is a sufficiently smooth function. The piecewise-linear functions we will present have constant and linear parts. The linear part depend on a parameter $a\in\R$ which will encode the expanding behaviour of the system. The nonsmooth bifurcations will arise when $\abs{a}>1$. For the four systems we consider, we show the existence of a curve of bifurcation $b^*(a)$ in the parameter plane where a nonsmooth bifurcation occurs. The curve of bifurcation can be expressed in terms of the absolute maximum or minimum of one of the invariant curves. Once we consider a particular $g$, we are able to explicitly determine the parameter values at which a nonsmooth bifurcation and an associated fractalization phenomenon occurs. 
Concretely, we show the existence of piecewise-linear versions of nonsmooth saddle-node, nonsmooth pitchfork, and nonsmooth period-doubling bifurcations. We want to make emphasis on the fact that, as far as we know, this is the first example of a family of quasiperiodically forced maps where it can be proved that there exists a nonsmooth period-doubling bifurcation. We avoid the case $a=1$ because, in this case, the linear part of the piecewise-linear systems becomes the identity map. In the non-linear case this behaviour is not generic.

We now summarize the key contributions and main results of the present study. As we will justify later, it is enough to study the bifurcation for $b>0$. We start with the case $\abs{a}>1$.
\begin{itemize}
    \item For $0<b<b^*(a)$, we establish the quantity of invariant or two periodic curves the systems have and we prove that the invariant or two periodic curves are piecewise of the same regularity as $g$. 
    \item For $b=b^*(a)$, we show that the systems undergo nonsmooth bifurcations. Usually, in computer simulations, the existence of a nonsmooth bifurcation is perceived through a fractalization phenomenon of the invariant curves of the system, when $b$ tends to $b^*(a)$. See, for example, \cite{Jorba2007Old} and the references therein. For more details we refer to Section \ref{fractalization_section}, where we define the fractalization of invariant curves. Unfortunately, this is not a characterization of the existence of a nonsmooth bifurcation (see Remark 2.4.3 and examples 2.4.4 - 2.4.5 in \cite{cobzacs2019lipschitz}). Nevertheless, under  suitable hypotheses, we are able to prove that a fractalization phenomenon is a necessary condition for the existence of a nonsmooth bifurcation, see Theorem \ref{weak_frac_thm}. In contrast, the computer simulations of the four piecewise-linear systems and the smooth counterparts barely show this fractalization phenomenon (compare Figures \ref{pic_perioddoubling} and \ref{pic_perioddoubling_smth}). This could be due to the fact that when two invariant or one invariant and one two-periodic curves collide, the set of angles corresponding to their intersection is residual but has zero measure, see Theorem \ref{big_thm}. The corresponding noncontinuous attracting invariant curve, which appears for this bifurcation, is called a Strange Nonchaotic Attractor: it is not the union of finitely many smooth manifolds and it has a zero Lyapunov exponent (due to the irrational rotation) and a negative Lyapunov exponent. In the case of having a noncontinuous two-periodic invariant curve, we say that the SNA is the union of the curve and its image. For an in-depth discussion of the different definitions of SNA, see \cite{Als2009Ondef}. 
    
    Additionally, if we fix $b=b^*(a)$, we can construct a sequence of continuous curves $\{\varphi_n\}_n$ which converge pointwise to the attracting curve. The graphs of $\{\varphi_n\}_n$ begin to exhibit wrinkling as $n\to\infty$. A curve is said to wrinkle when it exhibits progressively finer foldings, see Corollary \ref{char_corollary} for a more precise description of this behaviour. In Remark 3 of Keller's paper, the author raises the question of whether, when two invariant curves collide, the graph of one curve densely fills the region delimited by the two curves. Concerning this question, in Proposition \ref{densegraph} we prove that, at $b=b^*(a)$, our systems have an attracting invariant or two-periodic noncontinuous curve that is dense in a region with positive two-dimensional Lebesgue measure. Moreover, in Remark \ref{keller_remark} we comment that the proof of Proposition \ref{densegraph} also applies to Keller's systems. As a final conclusion to the study of these systems, in Theorem \ref{mon_characterization}, we give a criteria for a nonsmooth bifurcation in terms of the convergence of a particular sequence of curves.
    \item For $b>b^*(a)$, we prove that some of invariant or two periodic curves that exist for $0<b\leq b^*(a)$ disappear and the remaining curves are piecewise of the same regularity as $g$. 
\end{itemize}

For the case $\abs{a}<1$, using the Banach Fixed Point Theorem we prove the existence of a unique Lipschitz invariant curve, for all $b\in\R$.

Similar results have been proved for other classes of maps \cite{keller1996sna,Bjer2009sna,jorba2024nonsmooth,Fuh2016nonsmooth}. In the case of continuous dynamical systems, \cite{Dueñas2023Bif,Núñez2008Anon,Dueñas2025Sadd,Dueñas2024Gen} provide rigorous proofs of the existence of nonsmooth bifurcations of some families of nonautonomous scalar differential equations.

To finish with the introduction, we give a brief outline of the contents in the following sections.
Below in Section \ref{dyn_sys_sect}, we introduced the quasiperiodically forced systems. In Section \ref{invariant_curves_sect}, we present the invariant curves of the systems and some basic properties concerning these curves. In Section \ref{main_sect}, we state the main results of the work: we describe the bifurcation of the systems, we give conditions for which we can prove the nonsmoothness of the bifurcation, we describe the attracting behaviour of the attracting set for the parameters of bifurcation, we show the closure of an attracting curve is a set with positive two-dimensional Lebesgue measure for the parameters of bifurcation and, finally, we describe the fractalization process in these systems and we give a characterization of the noncontinuity of the attracting curve in terms of the uniform convergence of a sequence of continuous curves. In Section \ref{proof_sect}, we provide the detailed proofs for the system \eqref{dyn_sys_4}, with some comments for the other systems, and some other supplementary results about the systems and the invariant curves. In Sections \ref{main_sect}, \ref{section_b*}, \ref{three_curves_secion}, \ref{sect_after}, \ref{fractalization_section} we study the nonuniformly contractive case (when we can have an SNA) and in Section \ref{one_inv_curv} we consider the uniformly contracting case (when we have a unique attracting invariant curve). Concretely, in Section \ref{section_b*} there are the proofs of Theorems \ref{big_thm}, \ref{finite_num_iter}, \ref{densegraph} and Proposition \ref{lyap_exp}, these statements describe the properties of the systems at the bifurcation parameter. In Section \ref{three_curves_secion} there is the proof of Proposition \ref{prop_onlyonecurve} which describes the properties of the systems after the bifurcation parameter. In Section \ref{sect_after} there are the proofs of Proposition \ref{curve_between} and Theorem \ref{regularity_thm}, they describe the properties of the systems before the bifurcation parameter. In Section \ref{fractalization_section} there are the proofs of Theorems \ref{weak_frac_thm} and \ref{mon_characterization}, which show fractalization of the invariant curves under certain conditions. Finally, in Section \ref{one_inv_curv} there is the proof of Theorem \ref{thm_uniform_contraction_case} which describes the properties of the systems.

\subsection{The dynamical systems}\label{dyn_sys_sect}
Let $F\colon\R\times\T^1\to\R\times\T^1$ be a quasiperiodically forced map $(\Bar{x}, \Bar{\theta})=F(x,\theta)$ of the form
\begin{align*}
    \begin{cases}
        \Bar{x} = h(x) + b g(\theta),\\
        \Bar{\theta} = \theta + \omega \mod2\pi,
    \end{cases}
\end{align*}
where $b\in\R$ is a real parameter, $\omega\not\in2\pi\mathbb{Q}$, and $h\colon\R\to\R$ and $g\colon\T\to\R$ are continuous functions. Let $\T^\R$ denote the set of functions from $\T$ to $\R$. To such a system $F$, we associate the operator $\mathcal{F}\colon\T^\R\to \T^\R$ given by
\begin{align*}
    \mathcal{F}(\varphi)(\theta) = h(\varphi(\theta-\omega)) + b g(\theta-\omega).
\end{align*}
\begin{definition}\label{def_twi_inv}
    We say that a function $\varphi\in\T^\R$ is an \textbf{invariant curve} of a system of the form $F$ if $\mathcal{F}(\varphi)=\varphi$. A curve $\varphi\in\T^\R$ is a \textbf{two-periodic} curve of a system of the form $F$ if $\mathcal{F}^2(\varphi)=\varphi$ and $\mathcal{F}(\varphi)\not=\varphi$.
\end{definition}
Throughout the work we study four systems of this kind $F_i\colon\R\times\T^1\to\R\times\T^1$, for $i=1,2,3,4$. The maps $h_i$ appearing in the systems $F_i$ are defined piecewise with linear and constant components. We proceed to introduce the maps that will be the subject of our analysis in a more precise form.
\begin{itemize}
    \item The map $F_1$ is a quasiperiodically forced piecewise-linear version of a supercritical pitchfork bifurcation. It is given by
    \begin{align}\label{dyn_sys}
        \begin{cases}
            \Bar{x} = h_{1}(x) - b g(\theta),\\
            \Bar{\theta} = \theta + \omega \mod2\pi,
        \end{cases}
        &\hspace{1cm}
        h_{1}(x)= 
        \begin{cases}
            -\frac{\pi}{2} &\text{if }x\in(-\infty,-\frac{\pi}{2a}),\\
            ax & \text{if }x\in[-\frac{\pi}{2a},\frac{\pi}{2a}],\\
            \frac{\pi}{2} &\text{if }x\in(\frac{\pi}{2a},\infty),
        \end{cases}
    \end{align}
    with $a>0$.
    
    \item The map $F_2$ is a quasiperiodically forced piecewise-linear version of a subcritical pitchfork bifurcation. It is given by
    \begin{align}\label{dyn_sys_2}
        \begin{cases}
            \Bar{x} = h_{2}(x) + b g(\theta),\\
            \Bar{\theta} = \theta + \omega \mod2\pi,
        \end{cases}
        &\hspace{1cm}
        h_{2}(x)= 
        \begin{cases}
            a(x+\delta) &\text{if }x\in(-\infty,-\delta),\\
            0 & \text{if }x\in[-\delta,\delta],\\
            a(x-\delta) &\text{if }x\in(\delta,\infty),
        \end{cases}
    \end{align}
    with $a>0$ and $\delta>0$.

    \item The map $F_3$ is a quasiperiodically forced piecewise-linear version of a saddle-node bifurcation. It is given by
    \begin{align}\label{dyn_sys_3}
        \begin{cases}
            \Bar{x} = h_{3}(x) + b g(\theta),\\
            \Bar{\theta} = \theta + \omega \mod2\pi,
        \end{cases}
        &\hspace{1cm}
        h_{3}(x)= 
        \begin{cases}
            ax & \text{if }x>-\frac{1}{a},\\
            -1 &\text{if }x\leq-\frac{1}{a},
        \end{cases}
    \end{align}
    with $a>0$.

    \item The map $F_4$ is a quasiperiodically forced piecewise-linear version of a period-doubling bifurcation. It is given by
    \begin{align}\label{dyn_sys_4}
        \begin{cases}
            \Bar{x} = h_{4}(x) - b g(\theta),\\
            \Bar{\theta} = \theta + \omega \mod2\pi,
        \end{cases}
        &\hspace{1cm}
        h_{4}(x)= 
        \begin{cases}
            ax & \text{if }x>\frac{1}{a},\\
            1 &\text{if }x\leq \frac{1}{a},
        \end{cases}
    \end{align}
    with $a<0$.
\end{itemize}

Note that the invariant curves of the map $F_i$ are precisely the fixed points of the corresponding operator $\mathcal{F}_i$. The function $g$ is assumed to be a $C^{1+\tau}(\T)$ function, for some $\tau>0$.
Since our study is heavily dependent on the parameter $a$, we separate the study of the quasiperiodically forced maps in terms of this parameter. To show the existence of nonsmooth bifurcations in the four piecewise-linear systems, we will use similar arguments to the ones in \cite{jorba2024nonsmooth}. It is important to note that for $i=1,2,3$ the map $h_i$ is monotone increasing and for $i=4$ the map $h_4$ is monotone decreasing. Hence, the map $h_4^2$ is monotone increasing. We have introduced each dynamical system with a parameter $b$ which is any real number, but for the study of the system we will only consider $b>0$. Note that for the systems \eqref{dyn_sys}, \eqref{dyn_sys_2} and \eqref{dyn_sys_4} the change $(b,x,\Bar{x})\to(-b,-x,-\Bar{x})$ allows results to be extended from $b>0$ to $b<0$. In the following sections we introduce the study of system \eqref{dyn_sys_4}, which is the most different of the four families. The study of the other systems can be deduced from similar arguments. For completeness, the statements of the results are written for the four systems. In the proofs, we comment how to obtain the results for the rest of the systems. To be more precise, the system \eqref{dyn_sys} and \eqref{dyn_sys_2} are related to the system in \cite{jorba2024nonsmooth}. The system \eqref{dyn_sys_3} is a piecewise-linear approximation of the non-linear system
\begin{align*}
    \begin{cases}
        \Bar{x} = \exp(ax) - b g(\theta),\\
        \Bar{\theta} = \theta + \omega \mod2\pi,
    \end{cases}
\end{align*}
with $a>0$ and $b\in\R$.
The system \eqref{dyn_sys_4} is a piecewise-linear approximation of the non-linear system
\begin{align}\label{dyn_nonlin_period}
    \begin{cases}
        \Bar{x} = 1-\exp(ax) - b g(\theta),\\
        \Bar{\theta} = \theta + \omega \mod2\pi,
    \end{cases}
\end{align}
with $a>0$ and $b\in\R$. A preliminary analysis of these systems seems to indicate that their dynamical behaviour is similar to the piecewise-linear
ones. A more complete numerical study of these non-linear systems is work in progress.

Figure \ref{pic_perioddoubling_smth} shows the structure of the invariant set for the system \eqref{dyn_nonlin_period}. For the plots we have chosen $\omega=\pi(\sqrt{5}-1)$ and $a=3$. For the periodic function we use $g(\theta)=1+\cos(\theta)$. 
The green curves are two-periodic and the red one is a repelling invariant curve of the system \eqref{dyn_nonlin_period}. 

Figure \ref{pic_perioddoubling} shows the structure of the invariant set for the system \eqref{dyn_sys_4}. For the plots we have chosen $\omega=\pi(\sqrt{5}-1)$ and $a=-3$. For the periodic function we use again $g(\theta)=1+\cos(\theta)$. The green curves are two-periodic curves of the map $F_4$, and the red curve is the repelling invariant curve. In the bifurcation, we will prove that the two-periodic curves (green) and the invariant curve (red) are different. Moreover, the intersection set of the curves has zero Lebesgue measure. Therefore, numerical simulations cannot capture how the two-periodic curves densely fill the area between them. In Section \ref{main_sect}, we will prove that the curves densely fill the area and that there are no more invariant or two-periodic curves than those shown in the figure \ref{pic_perioddoubling}, see Corollary \ref{coro_numbercurves}.

\begin{figure}[htbp]
  \centering
  
  % First row
  \begin{subfigure}[b]{0.47\textwidth}
    \includegraphics[width=\textwidth]{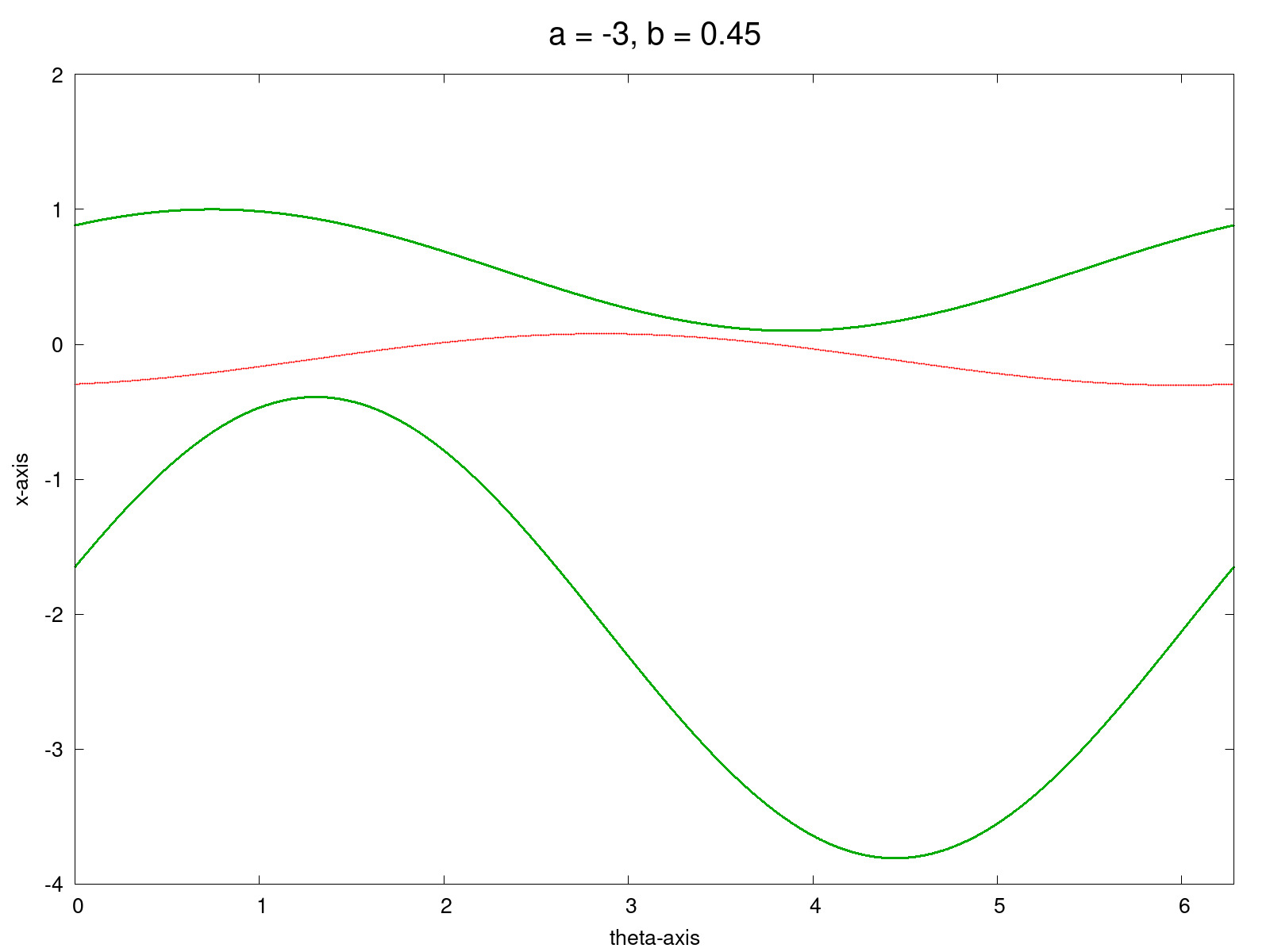}
  \end{subfigure}
  \hfill
  \begin{subfigure}[b]{0.47\textwidth}
    \includegraphics[width=\textwidth]{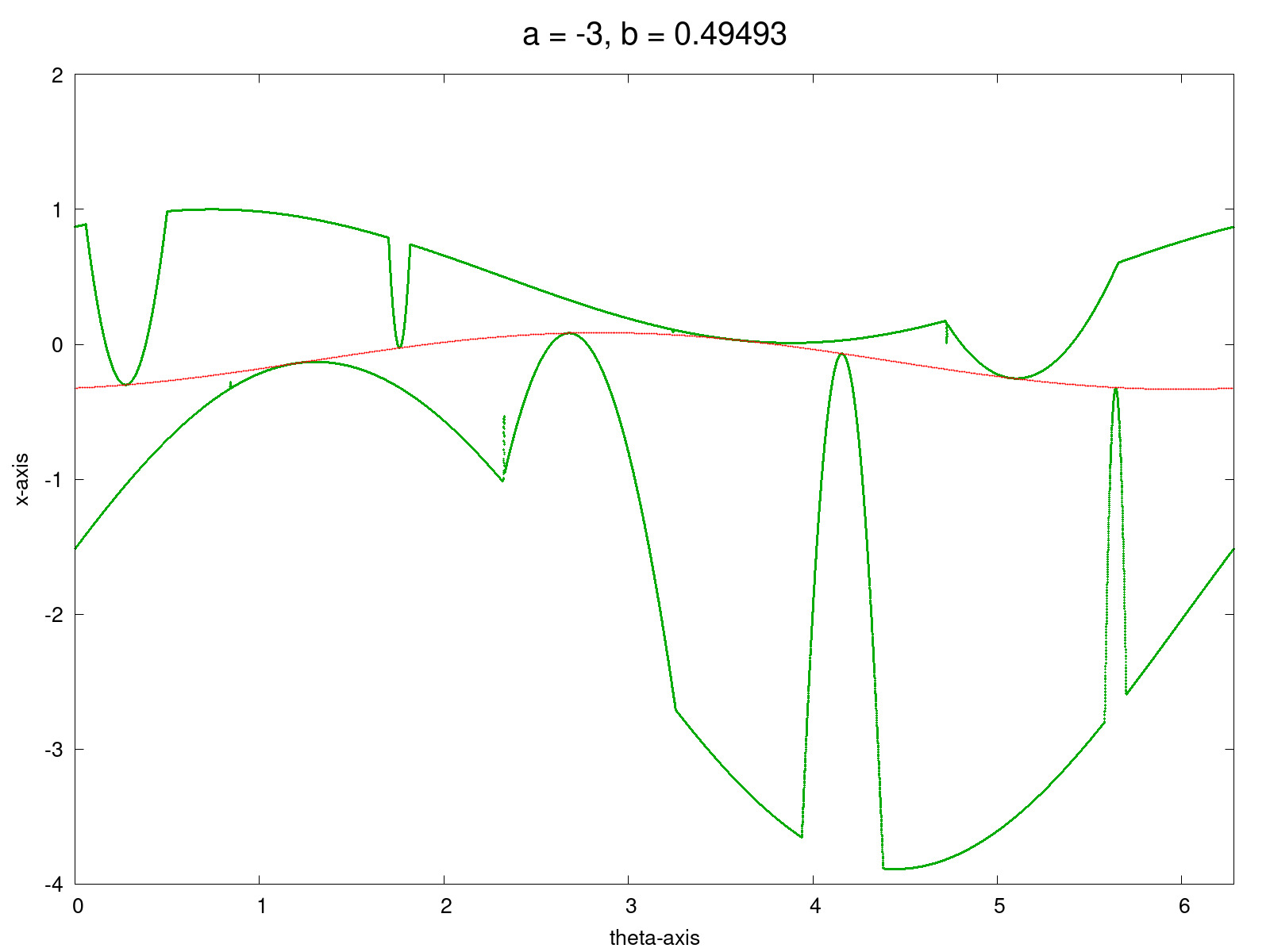}
  \end{subfigure}
  \hfill
  \begin{subfigure}[b]{0.47\textwidth}
    \includegraphics[width=\textwidth]{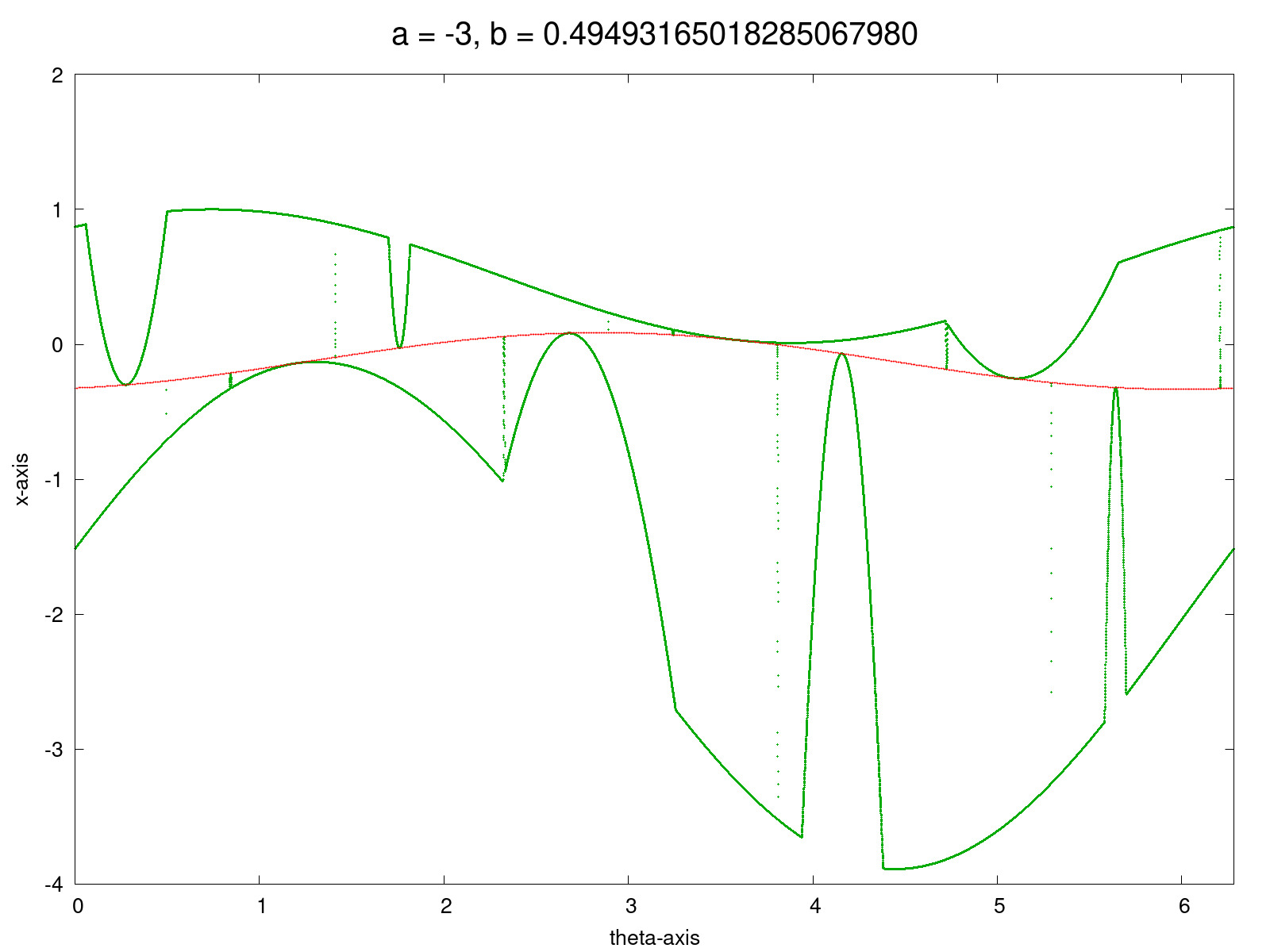}
  \end{subfigure}
  \hfill
  \begin{subfigure}[b]{0.47\textwidth}
    \includegraphics[width=\textwidth]{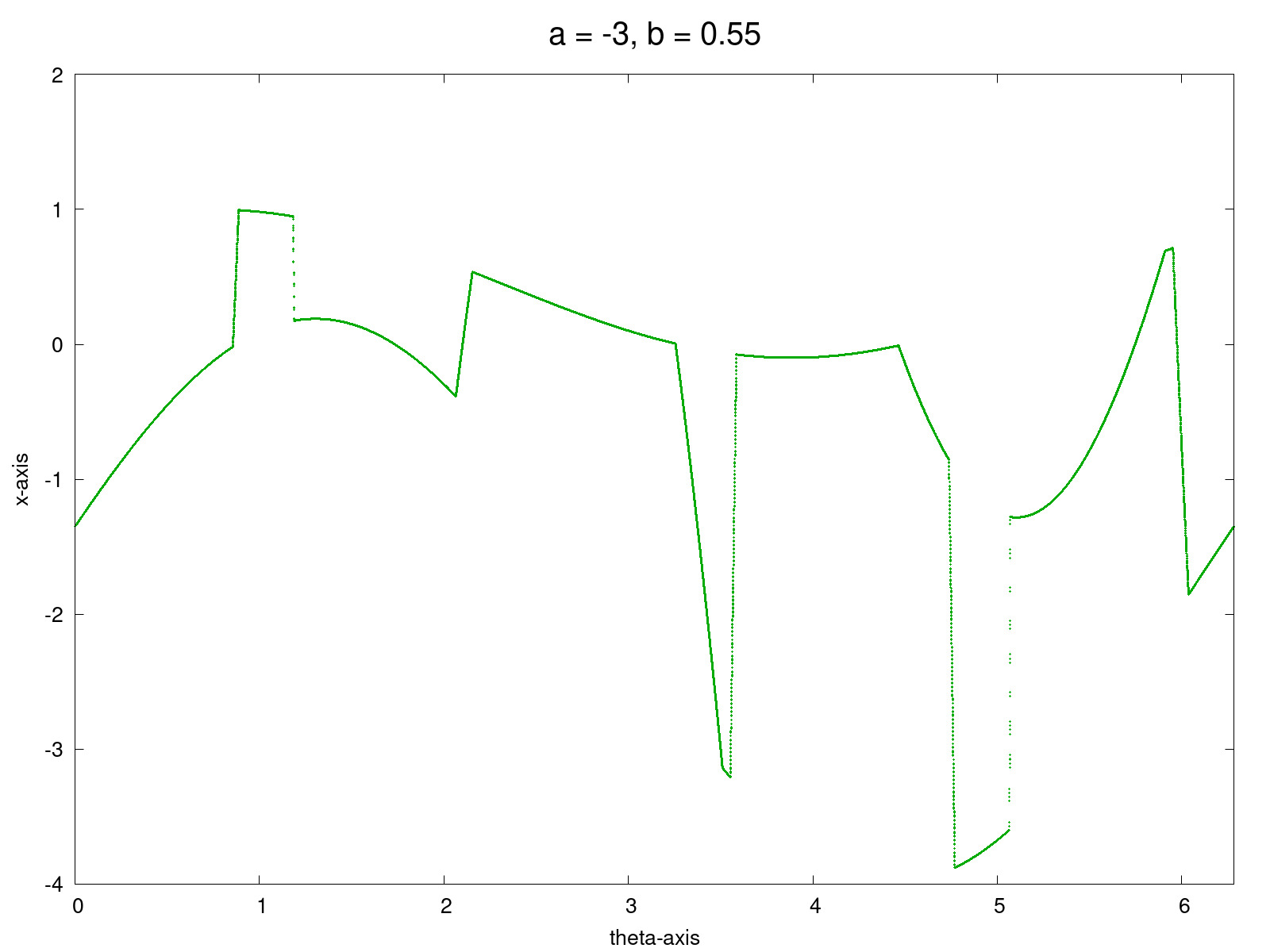}
  \end{subfigure}

  \caption{Two-periodic attracting (green) and repelling (red) curves of \eqref{dyn_sys_4} for fixed $a=-3$ and different values of $b$.}
  \label{pic_perioddoubling}
\end{figure}

\begin{figure}[htbp]
  \centering
  
  % First row
  \begin{subfigure}[b]{0.47\textwidth}
    \includegraphics[width=\textwidth]{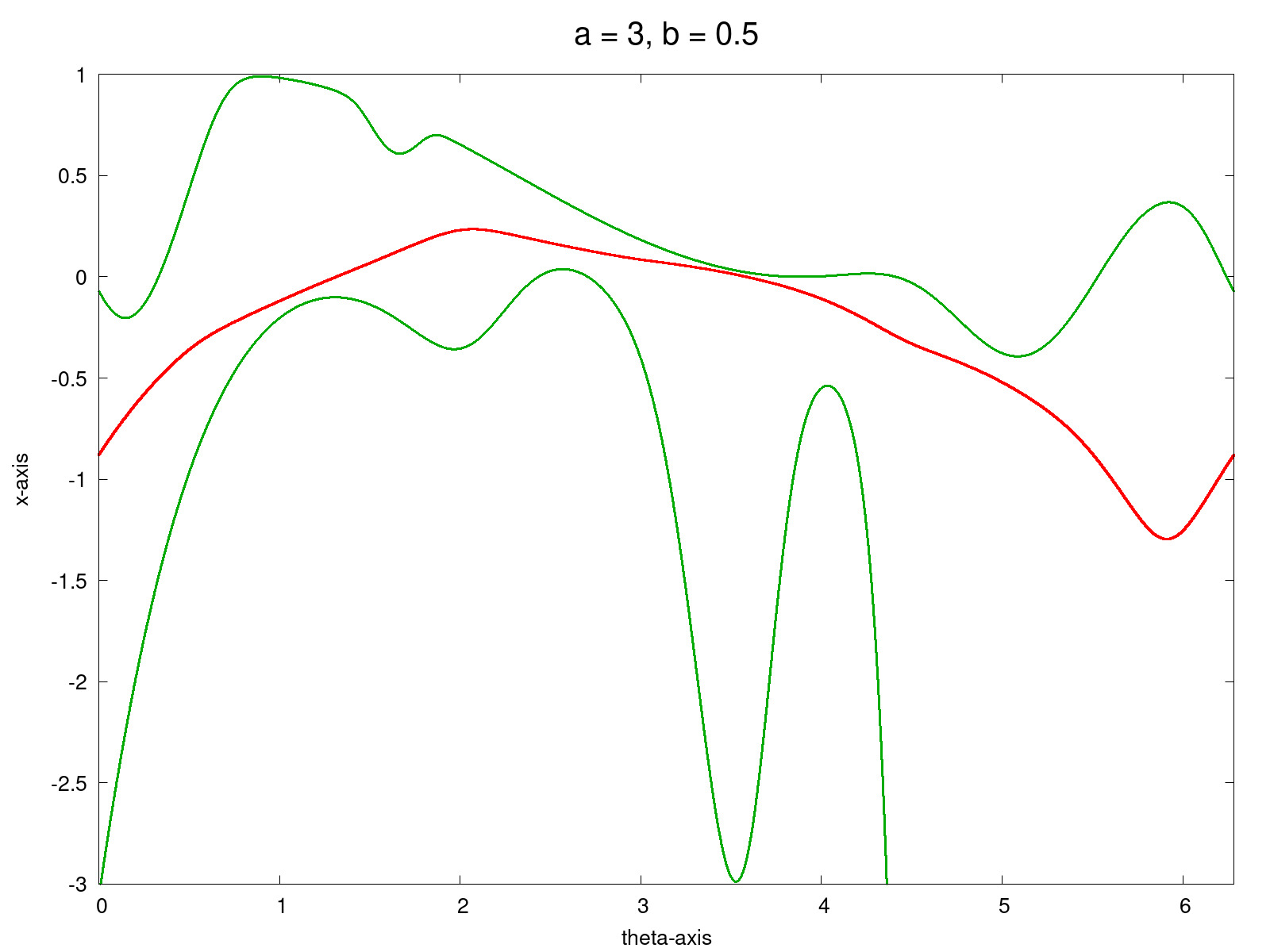}
  \end{subfigure}
  \hfill
  \begin{subfigure}[b]{0.47\textwidth}
    \includegraphics[width=\textwidth]{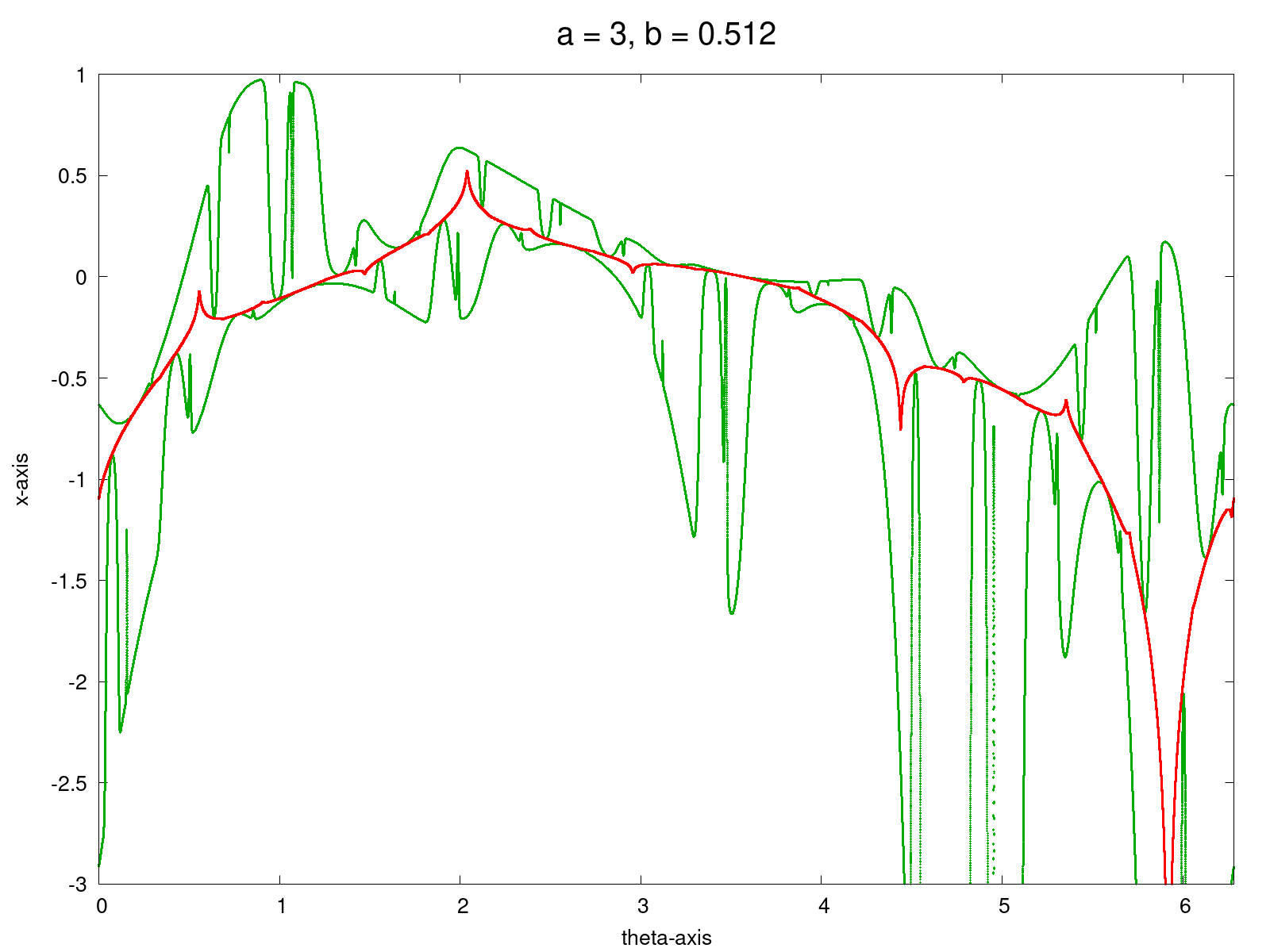}
  \end{subfigure}
  \hfill
  \begin{subfigure}[b]{0.47\textwidth}
    \includegraphics[width=\textwidth]{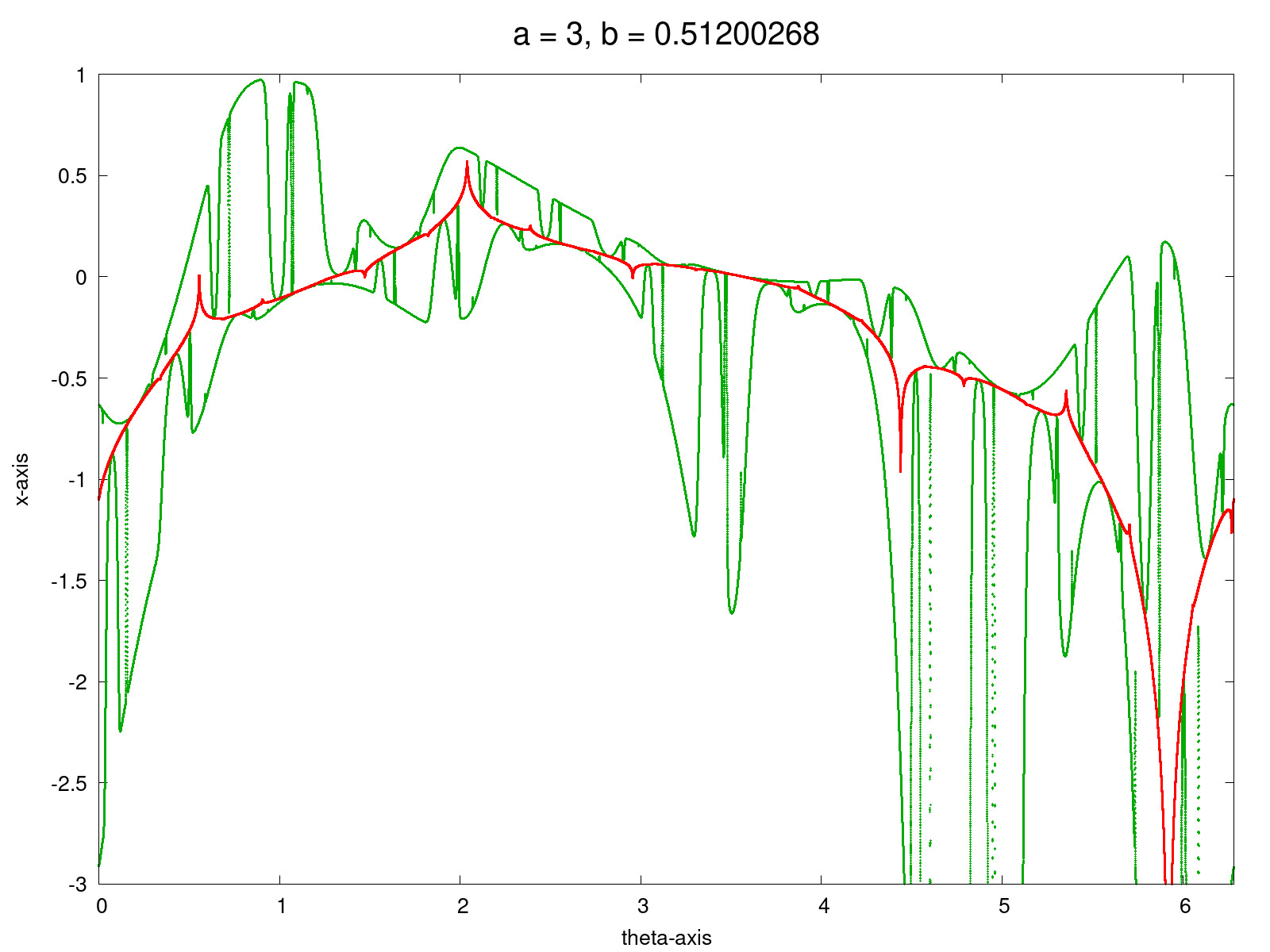}
  \end{subfigure}
  \hfill
  \begin{subfigure}[b]{0.47\textwidth}
    \includegraphics[width=\textwidth]{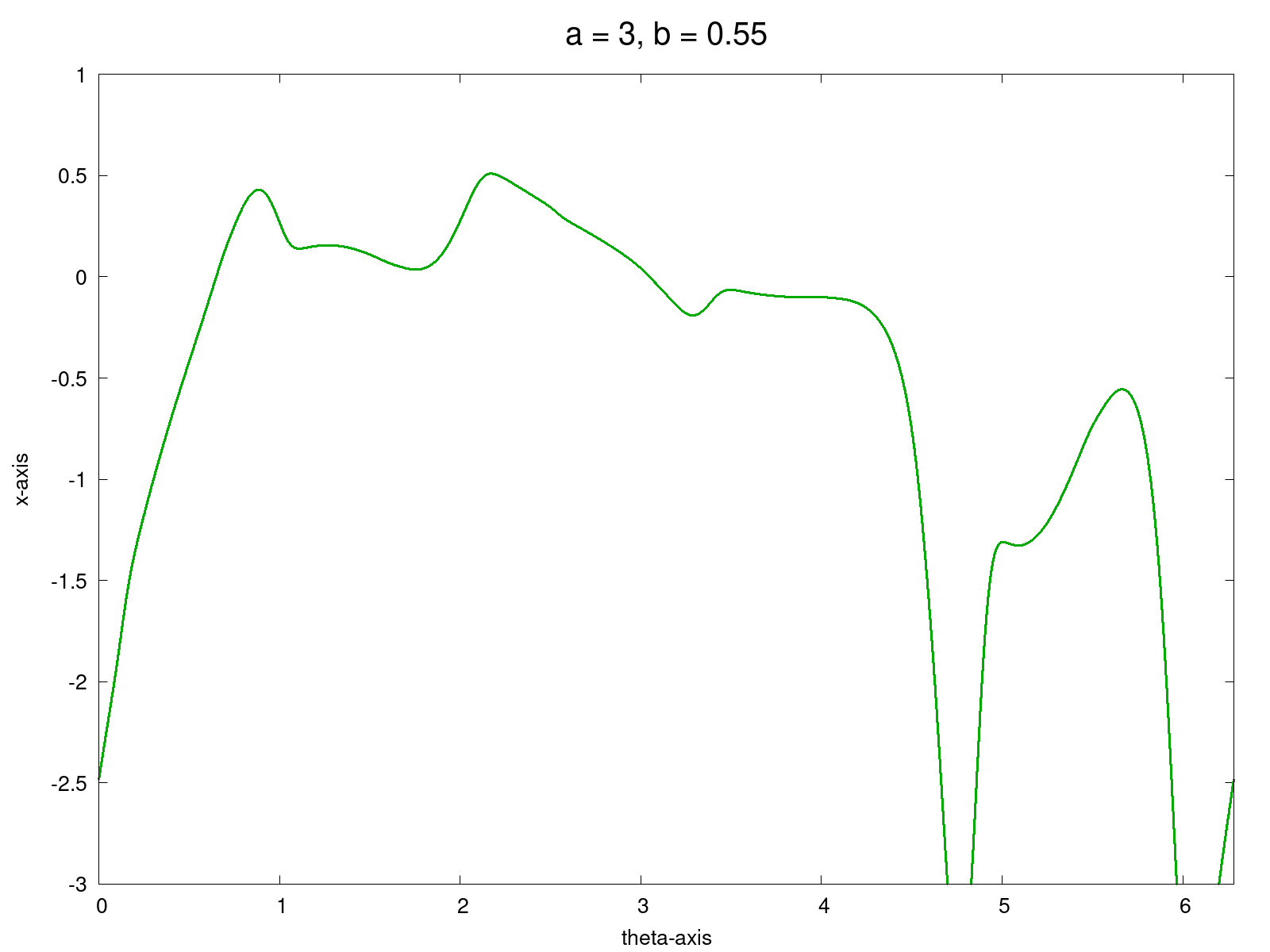}
  \end{subfigure}

  \caption{Attracting (green) and repelling (red) curves of the non-linear system \eqref{dyn_nonlin_period} for fixed $a=3$ and different values of $b$.}
  \label{pic_perioddoubling_smth}
\end{figure}

\section{The invariant curves of the systems}\label{invariant_curves_sect}
In this section, we study the non-uniform contraction (see Assumption \ref{assumptions_a}) case for the systems we have introduced previously. The arguments we use in the proofs follow the ideas from \cite{jorba2024nonsmooth}. As we have mentioned before, in order to avoid unnecessary repetition, we give only the complete proofs for the system \eqref{dyn_sys_4} and we comment briefly on how the proofs for the other systems may be carried out. Recall that, as we have mentioned in the introduction, we can reduce the study to the case with $b>0$. We start the section with an important assumption to which we will refer throughout the following sections.
\begin{assumption}\label{assumptions_a}
    For the systems \eqref{dyn_sys}, \eqref{dyn_sys_2} and \eqref{dyn_sys_3} we assume $a>1$ and for the \eqref{dyn_sys_4} we assume $a<-1$.
\end{assumption}
Notice that this assumption guarantees that there is no uniform contraction on each of the systems.
From now on we need to work with the invariant curves, so we start by computing them for small enough $b>0$. Assume that, for $i=1,3,4$, we have an invariant curve $\mu_i$ of the system $F_i$ such that the image of $\mu_i$ is on the linear part of $h_i$, then
\begin{align}\label{inv_eq_mu}
    \mu_i(\theta+\omega) &= h_{i}(\mu_i(\theta)) - b g(\theta).
\end{align}
Since we assume that the image of $\mu_i$ is on the linear part of $h_i$, if $g$ is a trigonometric polynomial, then equation \eqref{inv_eq_mu} has a solution $\mu_i$ which is a trigonometric polynomial. If $g$ can be expressed as a Fourier series, then the solution $\mu_i$ can be found using standard Fourier analysis. We explain how to obtain the solution after the proof of Lemma \ref{lemma_distance}. The curve $\mu_i$ is invariant as long as it stays on the region where $h_i$ is linear. Since every piecewise-linear $h_i$ has two or three different components, if $b=0$, there exists one invariant or two-periodic curve for each component. When $b>0$ is sufficiently small, the curves defined below are invariant or two-periodic. For larger $b$, they serve as a seed of an iteration procedure to find the invariant or two-periodic curves. We proceed to define the curves in question depending on the parameters $a$ and $b$: 
\begin{itemize}
    \item For \eqref{dyn_sys}, we consider the following curves
    \begin{align*}
        \tilde\varphi_1(\theta) &= \frac{\pi}{2} - b g(\theta-\omega),\\
        \tilde\gamma_1(\theta) &= -\frac{\pi}{2} - b g(\theta-\omega),
    \end{align*}
    and $\mu_1$ the continuous solution of the functional equation
    \begin{align*}
        \mu_1(\theta+\omega) = a\mu_1(\theta) - b g(\theta).
    \end{align*}
    \item For \eqref{dyn_sys_2}, we have two parts where the system is linear. Hence, we consider the following curves
    \begin{align*}
        \tilde\varphi_2(\theta) = b g(\theta-\omega),
    \end{align*}
    and $\mu_2$ and $\hat\mu_2$ the continuous solutions of the functional equations
    \begin{align*}
        \mu_2(\theta) = a(\mu_2(\theta-\omega)-\delta) +b g(\theta-\omega),\\
        \hat\mu_2(\theta) = a(\hat\mu_2(\theta-\omega)+\delta) +b g(\theta-\omega).
    \end{align*}
    \item For \eqref{dyn_sys_3}, we consider the following curves
    \begin{align*}
        \tilde\varphi_3(\theta) &= -1 + b g(\theta-\omega).
    \end{align*}
    and $\mu_3$ the continuous solution of the functional equation
    \begin{align*}
        \mu_3(\theta+\omega) = a\mu(\theta) + b g(\theta).
    \end{align*}
    \item For \eqref{dyn_sys_4}, we consider the following curves
    \begin{align*}
        \tilde\varphi_4(\theta) &= 1 - b g(\theta-\omega).
    \end{align*}
    and $\mu_4$ the continuous solution of the functional equation
    \begin{align}\label{mu_fun}
        \mu_4(\theta+\omega) = a\mu(\theta) - b g(\theta).
    \end{align}
\end{itemize}
If $b>0$ is small enough, these curves are invariant, for the systems \eqref{dyn_sys}, \eqref{dyn_sys_2} and \eqref{dyn_sys_3}.
In contrast, for the system \eqref{dyn_sys_4}, the curve $\mu_4$ is invariant but the curve $\tilde\varphi_4$ is not. What happens is that $\tilde\varphi_4$ is a two-periodic curve. In order to show that $\tilde\varphi_4$ is two-periodic curve, we need to consider a fixed curve for the map iterated two times. The two-iteration map of system \eqref{dyn_sys_4} has the following form
\begin{align}\label{dyn_sys_4_two}
    \begin{cases}
        \Bar{\Bar{x}} = h_a(h_{a}(x) - b g(\theta)) - b g(\theta+\omega),\\
        \Bar{\Bar{\theta}} = \theta + 2\omega \mod2\pi.
    \end{cases}
\end{align}
For small enough $b>0$ and $a<-1$, we have that $\tilde\varphi_4(\theta)>\frac{1}{a}$. Then,
\begin{align*}
    \pi_1\circ F_4(\tilde\varphi_4(\theta),\theta) = h_{4}(\tilde\varphi_4(\theta)) - b g(\theta) = a\tilde\varphi_4(\theta) - b g(\theta).
\end{align*}
Hence, for small enough $b>0$, we obtain that
\begin{align*}
    \pi_1\circ F_4^2(\tilde\varphi(\theta),\theta) = h_a(h_{a}(\tilde\varphi(\theta)) - b g(\theta)) - b g(\theta+\omega) = 1 - b g(\theta+\omega) = \tilde\varphi(\theta+2\omega).
\end{align*}

We know that the collision of invariant curves can lead to the noncontinuity (see \cite{keller1996sna,jorba2024nonsmooth}) of some of the colliding curves. The following result give an expression for the distance between two of the curves.
\begin{lemma}\label{lemma_distance}
    For each system, with Assumption \eqref{assumptions_a}, we have the following:
    \begin{itemize}
        \item For the system \eqref{dyn_sys} we have 
        \begin{align*}
            \min_{\theta\in\T}\left(\tilde\varphi_1(\theta)-\mu_1(\theta)\right)&=a\min_{\theta\in\T}\left(\frac{\pi}{2a}-\mu_1(\theta-\omega)\right),\\
            \min_{\theta\in\T}\left(\mu_1(\theta)-\tilde\gamma_1(\theta)\right)&=a\min_{\theta\in\T}\left(\mu_1(\theta-\omega)+\frac{\pi}{2a}\right).
        \end{align*}
        \item For the system \eqref{dyn_sys_2} we have
        \begin{align*}
            \min_{\theta\in\T}\left(\mu_2(\theta)-\tilde\varphi_2(\theta)\right)&=a\min_{\theta\in\T}\left(\mu_2(\theta-\omega)-\delta\right),\\
            \min_{\theta\in\T}\left(\hat\mu_2(\theta)-\tilde\varphi_2(\theta)\right)&=a\min_{\theta\in\T}\left(\hat\mu_2(\theta-\omega)+\delta\right).
        \end{align*}
        \item For the system \eqref{dyn_sys_3} we have
        \begin{align*}
            \min_{\theta\in\T}\left(\mu_3(\theta)-\tilde\varphi_3(\theta)\right)&=a\min_{\theta\in\T}\left(\mu_3(\theta-\omega)+\frac{1}{a}\right).
        \end{align*}
        \item For the system \eqref{dyn_sys_4} we have
        \begin{align*}
            \min_{\theta\in\T}\left(\tilde\varphi_4(\theta)-\mu_4(\theta)\right)&=a\min_{\theta\in\T}\left(\frac{1}{a}-\mu_4(\theta-\omega)\right).
        \end{align*}
    \end{itemize}
\end{lemma}
\begin{proof}
    For the system \eqref{dyn_sys_4}. From the functional equation \eqref{mu_fun} and the explicit expression of $\tilde\varphi_4$ we can see that
     \begin{align*}
        \tilde\varphi_4(\theta)-\mu_4(\theta)&=1-a\mu_4(\theta-\omega) = a\left(\frac{1}{a}-\mu_4(\theta-\omega)\right). 
    \end{align*}
    In the other systems, it can be used either the same argument with the corresponding functional equation or the explicit expression of the curve to obtain the result.
\end{proof}

We now focus on solving equation \eqref{inv_eq_mu} for the system \eqref{dyn_sys_4}. The reasoning for other systems follows by similar arguments.
If we consider that $g$ is a trigonometric polynomial with the expression $g(\theta)=\sum_{n=0}^N g_1(n)\sin(n\theta)+g_2(n)\cos(n\theta)$, then we know that the solution $\mu_4$ has the form $\mu_4(\theta)=\sum_{n=0}^N a_n\sin(n\theta)+b_n\cos(n\theta)$. Therefore, equation \eqref{inv_eq_mu} takes the form
\begin{align*}
    &\sum_{n=0}^N a_n\sin(n(\theta+\omega))+b_n\cos(n(\theta+\omega))\\
    &=a\sum_{n=0}^N a_n\sin(n\theta)+b_n\cos(n\theta) - b \sum_{n=0}^N g_1(n)\sin(n\theta)+g_2(n)\cos(n\theta).
\end{align*}
Recall the following trigonometric identities
\begin{align*}
    \sin(n\theta+n\omega)&=\sin(n\theta)\cos(n\omega)+\cos(n\theta)\sin(n\omega),\\
    \cos(n\theta+n\omega)&=\cos(n\theta)\cos(n\omega)-\sin(n\theta)\sin(n\omega).
\end{align*}
Therefore, we obtain the following system of equations for the unknowns $\{a_n,b_n\}_{n=0}^N$
\begin{align*}
    \begin{cases}
        a_n\cos(n\omega)-b_n\sin(n\omega)=a a_n-b g_1(n),\\
        a_n\sin(n\omega)+b_n\cos(n\omega)=a b_n-b g_2(n).
    \end{cases}
\end{align*}
This system in matrix form is
\begin{align}\label{system_for_invcurv}
    \begin{pmatrix}
        \cos(n\omega)-a & -\sin(n\omega)\\
        \sin(n\omega) & \cos(n\omega)-a
    \end{pmatrix}
    \begin{pmatrix}
        a_n\\
        b_n
    \end{pmatrix}
    =
    \begin{pmatrix}
        -b g_1(n)\\
        -b g_2(n)
    \end{pmatrix}.
\end{align}
The other systems give rise to similar linear systems for the coefficients $\{a_n,b_n\}_{n=0}^N$.
The determinant of the matrix in \eqref{system_for_invcurv} is $1-2a\cos(n\omega)+a^2>(a-1)^2$. Therefore, the solutions of the system \eqref{system_for_invcurv} are
\begin{align*}
    a_n = b\frac{
    \begin{vmatrix}
            -g_1(n) & -\sin(n\omega)\\
            -g_2(n) & \cos(n\omega)-a
        \end{vmatrix}
    }{1-2a\cos(n\omega)+a^2} = b\tilde{a}_n
    \text{ and }
    b_n = b\frac{
    \begin{vmatrix}
        \cos(n\omega)-a & -g_1(n)\\
        \sin(n\omega) & -g_2(n)
    \end{vmatrix}
    }{1-2a\cos(n\omega)+a^2} = b \tilde{b}_n.
\end{align*}

In the following result, we make explicit the dependence of $\mu_i$ with respect to the parameters $(a,b,\omega)$.
\begin{prop}\label{explicit_factor_mu}
    For each system, with Assumption \eqref{assumptions_a}, we have the following decomposition:
    \begin{itemize}
        \item For the system \eqref{dyn_sys}, if $\abs{\mu_1(a,b,\omega,\theta)}\leq \frac{\pi}{2a}$, then
        \begin{align*}
            \mu_1(a,b,\omega,\theta)=b\left(\sum_{n=0}^N \tilde{a}_n(a,\omega)\sin(n\theta)+\tilde{b}_n(a,\omega)\cos(n\theta)\right) = b\boldsymbol\mu_1(a,\omega,\theta).
        \end{align*}
    
        \item For the system \eqref{dyn_sys_2},
        \begin{itemize}
            \item if $\mu_2(a,b,\omega,\theta)\geq\delta$, then
            \begin{align*}
                \mu_2(a,b,\omega,\theta) &= b\boldsymbol\mu_2(a,\omega,\theta) - \frac{\delta}{a-1} - b\frac{g_1(0)}{a-1}.
            \end{align*}
            \item if $\hat\mu_2(a,b,\omega,\theta)\leq-\delta$, then
            \begin{align*}
                \hat\mu_2(a,b,\omega,\theta) &= b\boldsymbol{\hat\mu}_2(a,\omega,\theta) + \frac{\delta}{a-1} - b\frac{g_1(0)}{a-1}.
            \end{align*}
        \end{itemize}
        
        \item For the system \eqref{dyn_sys_3}, if $\mu_3(a,b,\omega,\theta)\geq \frac{-1}{a}$, then
        \begin{align*}
            \mu_3(a,b,\omega,\theta) = b\boldsymbol\mu_3(a,\omega,\theta).
        \end{align*}
    
        \item For the system \eqref{dyn_sys_4}, if $\mu_4(a,b,\omega,\theta)\geq \frac{1}{a}$, then
        \begin{align*}
            \mu_4(a,b,\omega,\theta) = b\boldsymbol\mu_4(a,\omega,\theta).
        \end{align*}
    \end{itemize}
\end{prop}

In other words, this result shows that the curve $\mu_i$ can be decomposed as the multiplication of $b$ with another function $\boldsymbol\mu_i$ independent of $b$ (plus a constant in system \eqref{dyn_sys_2}). Moreover, for the system \eqref{dyn_sys_4}, the coefficients $\tilde{a}_n$ and $\tilde{b}_n$ are
\begin{align*}
    \tilde{a}_n = \frac{
    \begin{vmatrix}
            -g_1(n) & -\sin(n\omega)\\
            -g_2(n) & \cos(n\omega)-a
        \end{vmatrix}
    }{1-2a\cos(n\omega)+a^2}
    \text{ and }
    \tilde{b}_n = \frac{
    \begin{vmatrix}
        \cos(n\omega)-a & -g_1(n)\\
        \sin(n\omega) & -g_2(n)
    \end{vmatrix}
    }{1-2a\cos(n\omega)+a^2}.
\end{align*}
Note that we extend the argument to any $g\in L^2(\T)$. Then, the relation in \eqref{system_for_invcurv} shows the relation between the Fourier coefficients of $g$ and $\mu_4$. Now, as mentioned in the introduction, we assume $g\in C^{1+\tau}$ for some $\tau>0$.
\begin{thm}\label{regularity_mu}
    Let $a$ satisfy the condition in \eqref{assumptions_a}. For every system $F_i$ the repelling invariant curve $\mu_i$ is $C^0$.
\end{thm}
\begin{proof}
    Consider the case for the system \eqref{dyn_sys_4}. Denote by $C(a,\omega,n)=1/(1-2a\cos(n\omega)+a^2)$, $c_1(a,\omega,n)=\cos(n\omega)-a$ and $c_2(a,\omega,n)=\sin(n\omega)$. Since $\abs{a}>1$, then 
    \begin{align*}
        \abs{C(a,\omega,n)}<\frac{1}{(a -1)^2}.
    \end{align*}
    Since $g\in C^{1+\tau}$ with $\tau>0$, then the Fourier coefficients of $\mu$ satisfy that
    \begin{align*}
        \abs{\tilde{a}_n}&\leq \abs{C(a,\omega,n)}(\abs{c_1(a,\omega,n)} \abs{g_1(n)} + \abs{c_2(a,\omega,n)} \abs{g_2(n)}) = O\left(1/\abs{n}^{1+\tau}\right),\\
        \abs{\tilde{b}_n}&\leq \abs{C(a,\omega,n)}(\abs{c_2(a,\omega,n)} \abs{g_1(n)} + \abs{c_1(a,\omega,n)} \abs{g_2(n)} = O\left(1/\abs{n}^{1+\tau}\right).
    \end{align*}
    Hence, the Fourier series representing $\mu_4$ converges absolutely and uniformly. We can do similarly with the other systems.
    For more details about the decay of Fourier coefficients see \cite{katz2004ha,Stein2003Four}.
\end{proof}
\begin{remark}\label{remarks_curves}
    \begin{enumerate}[label=(\roman*)]
        \item\label{regularity_remark} Note that using the same Fourier estimates in Theorem \ref{regularity_mu} we obtain:
        \begin{itemize}
            \item If $g$ is $C^{\alpha+\tau}$ for $\alpha\in\N$ such that $\alpha>0$, then $\mu_i$ are $C^{\alpha-1}$.
            \item If $g$ is analytic, then $\mu_i$ are analytic.
        \end{itemize}
        \item For the systems \eqref{dyn_sys} and \eqref{dyn_sys_2}, the distance between images of the curves in the same fibre is positive. For the system \eqref{dyn_sys_4}, we are able to obtain an inequality.
        \begin{itemize}
            \item For the system \eqref{dyn_sys} the distance between images of the same fibre $\theta\in\T$ is
            \begin{align*}
                \tilde\varphi_1(\theta)-\tilde\gamma_1(\theta)=\pi.
            \end{align*}
            \item For the system \eqref{dyn_sys_2} we need to work a bit. If we look at the functional equation of $\mu_2$ and $\hat\mu$, and we compute the relation of the Fourier coefficients, we can see that the coefficients of $\mu_2$ and $\hat\mu$ only differ on $n=0$. Therefore, with these computations we can see that the distance between images of the same fibre $\theta\in\T$ is
            \begin{align*}
                \mu_2(\theta)-\hat\mu_2(\theta)=\frac{2\delta a}{(a-1)}.
            \end{align*}
            \item For the system \eqref{dyn_sys_4} we can obtain the following inequality,
            \begin{align*}
                \mathcal{F}_4(\tilde\varphi_4)(\theta)\leq \tilde\varphi_4(\theta).
            \end{align*}
        \end{itemize}
    \end{enumerate}
    
\end{remark}
The following Proposition shows that, if $g$ is $\pi$-antisymmetric or positive, then we have additional properties for the system \eqref{dyn_sys}.

\begin{prop}\label{g_cases}
    Let $a>1$ and $b>0$. For the system \eqref{dyn_sys} we have the following:
    \begin{enumerate}
        \item If $g(\theta+\pi)=-g(\theta)$ for all $\theta\in\T$, then $\tilde\varphi_1$ and $\tilde\gamma_1$ approach $\mu_1$ at the same time as $b$ grows.
        \item If $g(\theta)\geq0$ for all $\theta\in\T$, then $\tilde\gamma_1$ and $\mu_1$ never intersect and $\tilde\varphi_1$ approach $\mu_1$ as $b$ grows.
    \end{enumerate}
\end{prop}
\begin{proof}
    \begin{enumerate}
        \item In this case $-\tilde\varphi_1(\theta+\pi) = -\frac{\pi}{2} + b g(\theta+\pi) = -\frac{\pi}{2} - b g(\theta) = \tilde\gamma_1(\theta)$ and also that $b_0$ the independent term should be zero. The last is because the system for $b_0$ results in $(1-a)b_0 = -b g_2(0) = 0$ and we have by hypothesis that $a>1$. Indeed, using the trigonometric identity for the sum we can compute that the Fourier coefficients of $g$ verify that
        \begin{align*}
            g_1(n)\sin(n(\theta+\pi))+g_2(n)\cos(n(\theta+\pi)) &= -\left( g_1(n)\sin(n\theta)+g_2(n)\cos(n\theta)\right),\\
            g_1(n)(-1)^n\sin(n\theta)+g_2(n)(-1)^n\cos(n\theta) &= -\left( g_1(n)\sin(n\theta)+g_2(n)\cos(n\theta)\right).
        \end{align*}
        So for $n$ even we have $g_1(n)=0=g_2(n)$. Therefore, considering the system for $(a_n,b_n)$ we get that for $n$ even $(a_n,b_n)=(0,0)$.
        Hence, again using the trigonometric identity for the sum we have that
        \begin{align*}
            \mu_1(\theta+\pi)&=\sum_{k=1}^\infty a_{2k+1}\sin((2k+1)(\theta+\pi))+b_{2k+1}\cos((2k+1)(\theta+\pi))\\
            &= -\sum_{k=1}^\infty a_{2k+1}\sin((2k+1)\theta)+b_{2k+1}\cos((2k+1)\theta) = -\mu_1(\theta).
        \end{align*}
        Therefore, the statement is a consequence of Lemma \eqref{lemma_distance},
            \begin{align}\label{symmetriccase_1_2}
                \min_{\theta\in\T}\left(\tilde\varphi_1(\theta)-\mu_1(\theta)\right)
                 = \min_{\theta\in\T}\left(\mu_1(\theta+\pi)-\tilde\gamma_1(\theta+\pi)\right).
            \end{align}

        \item Since $g(\theta)\geq0$ for all $\theta\in\T$, then we have that the map applied to the a curve $Z\colon\T\to\R$ such that $-\frac{\pi}{2}<Z(\theta)<0$ has image
            \begin{align*}
                \mathcal{F}(Z)(\theta) &= h_{a}(Z(\theta-\omega)) - b g(\theta-\omega) = a Z(\theta-\omega) - b g(\theta-\omega)\\
                &< Z(\theta-\omega) - b g(\theta-\omega) \leq Z(\theta-\omega).
            \end{align*}
        Since all the iterates of $\{x<0\}$ converge to $\tilde\gamma_1$ we obtain that the invariant curve $\mu_1$ is on $\{x\geq0\}$ for any $b>0$. Hence, by the expression of the invariant curve $\mu_1$ we get that $\tilde\varphi_1$ approach $\mu_1$ as $b$ grows.
    \end{enumerate}
\end{proof}

As a consequence of Proposition \eqref{explicit_factor_mu} and Lemma \eqref{lemma_distance}, we obtain that there exists a unique value of the parameter $b$, denoted by $b^*(a)$, for which there exists a collision between some of the invariant curves.

\begin{corollary}
    Let $a$ satisfy the condition in \eqref{assumptions_a} and $b\leq b^*(a)$. For each system there exists a unique value $b=b^*(a)$ for which intersection of the graphs of the invariant curves is not empty. 
    \begin{itemize}
        \item For the system \eqref{dyn_sys}, if we denote by $\theta_M$ a global maximum and $\theta_m$ a global minimum of $\boldsymbol\mu_1$, then $b^*(a)$ verifies that
        \begin{align*}
            b^*(a)\max\{\abs{\boldsymbol\mu_1(\theta_m)},\abs{\boldsymbol\mu_1(\theta_M)}\} = \frac{\pi}{2a}.
        \end{align*}
        Moreover, for $b = b^*(a)$ if $\max\{\abs{\boldsymbol\mu_1(\theta_m)},\abs{\boldsymbol\mu_1(\theta_M)}\}=\abs{\boldsymbol\mu_1(\theta_M)}$. Then, the curves $\tilde\varphi_1$ and $\mu_1$ intersect. If $\max\{\abs{\boldsymbol\mu_1(\theta_m)},\abs{\boldsymbol\mu_1(\theta_M)}\}=\abs{\boldsymbol\mu_1(\theta_m)}$, then the curves $\tilde\gamma_1$ and $\mu_1$ intersect.
        \item For the system \eqref{dyn_sys_2}, if we denote by $\theta_m$ a global minimum of $\boldsymbol\mu_2$ and $\theta_M$ a global maximum of $\boldsymbol{\hat\mu}_2$, then $b^*(a)$ verifies that
        \begin{align*}
            b^*(a)\max\left(\abs*{\boldsymbol\mu_2(\theta_m) - \frac{g_1(0)}{a-1}}, \abs*{\boldsymbol{\hat\mu}_2(\theta_m) - \frac{g_1(0)}{a-1}} \right)= \frac{a\delta}{a-1}
        \end{align*}
        \item For the system \eqref{dyn_sys_3}, if we denote by $\theta_m$ a global minimum of $\boldsymbol\mu_3$ then $b^*(a)$ verifies that
        \begin{align*}
            b^*(a)\boldsymbol\mu_3(\theta_m) = -\frac{1}{a}.    
        \end{align*}
        \item For the system \eqref{dyn_sys_4}, if we denote by $\theta_m$ a global minimum of $\boldsymbol\mu_4$, then $b^*(a)$ verifies that
        \begin{align*}
            b^*(a)\boldsymbol\mu_4(\theta_m) = \frac{1}{a}.    
        \end{align*}
    \end{itemize}
\end{corollary}
\begin{proof}
    Consider the system \eqref{dyn_sys_4}. From Lemma \eqref{lemma_distance}, we obtain that
    \begin{align*}
            \min_{\theta\in\T}\left(\tilde\varphi_4(\theta)-\mu_4(\theta)\right)&=a\min_{\theta\in\T}\left(\frac{1}{a}-\mu_4(\theta-\omega)\right).
        \end{align*}
    Using the decomposition in Proposition \ref{explicit_factor_mu} we obtain the expression for $b^*(a)$. The same argument proves the expressions for the other systems.
\end{proof}
\begin{remark}
    Note that Lemma \ref{lemma_distance} provides a sufficient condition to have a non-empty intersection between the curves. In particular, when we have the symmetry condition, $g(\theta+\pi)=-g(\theta)$ for all $\theta\in\T$, we have non-empty intersection between three curves.
\end{remark}
Now, using the maps $\mathcal{F}_i$ and the curves $\tilde\varphi_i$, we introduce an iteration procedure which let us obtain an invariant or two-periodic curve of the system.
\begin{lemma}\label{montonicity_lemma}
    Let $a$ satisfy Assumption \eqref{assumptions_a}. For each system we define two sequences of monotone curves. 
    \begin{itemize}
        \item For the system \eqref{dyn_sys}, we define the following functions $\varphi_0=\tilde\varphi_1$ and $\lambda_0 = \varphi_0-\mu_1$, and for all $n\geq0$ we set
        \begin{align*}
            \varphi_{n+1}(\theta) &:= \mathcal{F}_1(\varphi_n)(\theta) = h_1(\varphi_n(\theta-\omega)) -b g(\theta-\omega),\\
            \lambda_{n+1}(\theta) &:= h_1(\lambda_n(\theta-\omega) + \mu_1(\theta-\omega)) - a\mu_1(\theta-\omega).
        \end{align*}
        Then, we have that
        \begin{enumerate}
            \item $\lambda_n = \varphi_n-\mu_1$ for all $n\geq0$.
            \item If $0<b\leq b^*$, then $\lambda_n\geq 0$ for all $n\geq0$.
            \item The sequence of functions $\{\varphi_n\}_n$ and $\{\lambda_n\}_n$ are decreasing.
        \end{enumerate}
        \item For the system \eqref{dyn_sys_2}, we define the following functions $\varphi_0=\tilde\varphi_2$ and $\lambda_0 = \varphi_0-\mu_2$, and for all $n\geq0$ we set
        \begin{align*}
            \varphi_{n+1}(\theta) &:= \mathcal{F}_2(\varphi_n)(\theta) = h_{2}(\varphi_n(\theta-\omega)) + b g(\theta-\omega),\\
            \lambda_{n+1}(\theta) &:= a(\mu_2(\theta-\omega)-\delta) - h_{2}(\mu_2(\theta-\omega)-\lambda_n(\theta-\omega)).
        \end{align*}
        Then, we have that
        \begin{enumerate}
            \item $\lambda_n = \mu_2-\varphi_n$ for all $n\geq0$.
            \item If $0\leq b\leq b^*(a)$, then $\lambda_n\geq 0$ for all $n\geq0$.
            \item The sequence of functions $\{\varphi_n\}_n$ is monotone increasing and $\{\lambda_n\}_n$ is monotone decreasing.
        \end{enumerate}
        \item For the system \eqref{dyn_sys_3}, we define the following functions $\varphi_0=\tilde\varphi$ and $\lambda_0 = \mu_3-\varphi_0$, and for all $n\geq0$ we set
        \begin{align*}
            \varphi_{n+1}(\theta) &:= \mathcal{F}_3(\varphi_n)(\theta) = h_{3}(\varphi_n(\theta-\omega)) + b g(\theta-\omega),\\
            \lambda_{n+1}(\theta) &:= a\mu_3(\theta-\omega) - h_{3}(\mu_3(\theta-\omega)-\lambda_n(\theta-\omega)).
        \end{align*}
        Then, we have that
        \begin{enumerate}
            \item $\lambda_n = \mu_3-\varphi_n$ for all $n\geq0$.
            \item If $0\leq b\leq b^*(a)$, then $\lambda_n\geq 0$ for all $n\geq0$.
            \item The sequence of functions $\{\varphi_n\}_n$ is monotone increasing and $\{\lambda_n\}_n$ is monotone decreasing.
        \end{enumerate}
        \item For the system \eqref{dyn_sys_4}, we define the following functions $\varphi_0=\tilde\varphi$ and $\lambda_0 = \varphi_0-\mu_4$, and for all $n\geq0$ we set
        \begin{align*}
            \varphi_{n+1}(\theta) &:= \mathcal{F}_4^2(\varphi_n)(\theta) = h_4(h_{4}(\varphi_n(\theta-2\omega)) - b g(\theta-2\omega)) - b g(\theta-\omega),\\
            \lambda_{n+1}(\theta) &:= h_4(h_{4}(\mu_4(\theta-2\omega)+\lambda_n(\theta-2\omega)) - b g(\theta-2\omega)) - a\mu_4(\theta-\omega).
        \end{align*}
        Then, we have that
        \begin{enumerate}
            \item $\lambda_n = \varphi_n-\mu_4$ for all $n\geq0$.
            \item If $0\leq b\leq b^*(a)$, then $\lambda_n\geq 0$ for all $n\geq0$.
            \item The sequence of functions $\{\varphi_n\}_n$ is monotone decreasing and $\{\lambda_n\}_n$ is monotone decreasing.
        \end{enumerate}
    \end{itemize}
\end{lemma}
\begin{proof}
    We prove the three statements for the system \eqref{dyn_sys_4}. The other ones follow with similar arguments. In the case of the system \eqref{dyn_sys_4}, in order to prove the item (3) it is important to recall that $h_4$ is decreasing and hence $h_4\circ h_4$ is increasing.
    \begin{enumerate}
        \item By definition we have that it is true for $n=0$. So we proceed by induction, we assume that $\lambda_n = \varphi_n-\mu_4$ so
        \begin{align*}
            \lambda_{n+1}(\theta) &= h_4(h_{4}(\mu_4(\theta-2\omega)+\lambda_n(\theta-2\omega)) - b g(\theta-2\omega)) - a\mu_4(\theta-\omega)\\
            &= h_4(h_{4}(\varphi_n(\theta-2\omega)) - b g(\theta-2\omega)) - \mu_4(\theta) - b g(\theta-\omega)\\
            &= \varphi_{n+1}(\theta)-\mu_4(\theta).
        \end{align*}

        \item We know that for $0\leq b\leq b^*(a)$ the curves $\mu_4$ and $\varphi_0$ do not intersect, and that $\{\varphi_n\}_n$ is monotone decreasing and bounded by $\mu_4$. Hence, as $\mu_4$ is invariant $\lambda_n = \varphi_n-\mu_4\geq0$. If $b = b^*(a)$, we have that they intersect only in one point, so by the same argument we have that $\lambda_n\geq0$ for all $n\geq0$.
        
        \item We compute that for all $\theta\in\T$
        \begin{align*}
            \varphi_{1}(\theta) = h_4(h_{4}(\varphi_0(\theta-2\omega)) - b g(\theta-2\omega)) - b g(\theta-\omega) \leq 1 - b g(\theta-\omega) = \varphi_0(\theta)
        \end{align*}
        Assume by induction that $\varphi_{n}\leq\varphi_{n-1}$, since $h_a$ is a decreasing function, we get that for all $\theta\in\T$
        \begin{align*}
            h_{4}(\varphi_n(\theta-2\omega)) - b g(\theta-2\omega)\geq h_{4}(\varphi_{n-1}(\theta-2\omega)) - b g(\theta-2\omega).
        \end{align*}
        Hence, for all $\theta\in\T$
        \begin{align*}
            \varphi_{n+1}(\theta) &= h_4 h_{4}(\varphi_n(\theta-2\omega)) - b g(\theta-2\omega)) - b g(\theta-\omega) \\
            &\leq h_4(h_{4}(\varphi_{n-1}(\theta-2\omega)) - b g(\theta-2\omega)) - b g(\theta-\omega) = \varphi_{n}(\theta).
        \end{align*}
        Since $\lambda_n = \varphi_n-\mu_4$ we obtain that $\{\lambda_n\}_n$ is monotone decreasing.
    \end{enumerate}
\end{proof}

To make more simple the notation, we will denote by $\{\varphi_n\}_n$ all the sequences generated by the map $\mathcal{F}_i$, for $i=1,2,3,4$. This is because the sequence $\varphi_n$ will play the same role for $i=1,2,3,4$. We do the same simplification of notation with the sequence $\{\lambda_n\}_n$. 

Since, for $i=1,2,3,4$, either $h_i$ or $h_i^2$ is monotone increasing, the sequences $\{\varphi_n\}_n$ and $\{\lambda_n\}_n$ converge pointwise to some curve $\varphi_\infty$ and $\lambda_\infty$ respectively. 

We used that a monotone bounded sequence of continuous curves has to converge pointwise to a semicontinuous curve. For every piecewise-linear system, the final goal is to prove the noncontinuity of $\varphi_\infty$, at the parameter $b=b^*(a)$.

The following proposition will allow us to define an invariant compact set.
\begin{prop}\label{attract_prop}
    Let $a$ satisfy Assumption \eqref{assumptions_a} and $b\leq b^*(a)$.
    \begin{itemize}
        \item For the system \eqref{dyn_sys}, we define the set
        \begin{align*}
            A_{+}:=\{(x,\theta)\in\R\times\T\mid\mu_1(\theta)\leq x \leq\varphi_0(\theta)\}.
        \end{align*}
        Then $F_{1}(A_{+})\subset A_{+}$.
        \item For the system \eqref{dyn_sys_2}, we define the set
        \begin{align*}
            A_{+}:=\{(x,\theta)\in\R\times\T\mid\varphi_0(\theta)\leq x \leq\mu_2(\theta)\}.
        \end{align*}
        Then $F_{2}(A_{+})\subset A_{+}$.
        \item For the system \eqref{dyn_sys_3}, we define the set
        \begin{align*}
            A_{+}:=\{(x,\theta)\in\R\times\T\mid\varphi_0(\theta)\leq x \leq\mu_3(\theta)\}.
        \end{align*}
        Then $F_{3}(A_{+})\subset A_{+}$.
        \item For the system \eqref{dyn_sys_4}, we define the set
        \begin{align*}
            A_{+}:=\{(x,\theta)\in\R\times\T\mid\mu_4(\theta)\leq x\leq \varphi_0(\theta)\}.
        \end{align*}
        Then $F_{4}^2(A_{+})\subset A_{+}$.
    \end{itemize}
\end{prop}
\begin{proof}
    First consider the case for the system \eqref{dyn_sys_4}. Assume $(x,\theta)\in A_{+}$. Since $h_4$ is decreasing we have that
    \begin{align*}
        h_4\circ\mu_4(\theta) &\geq h_4(x) \geq h_4\circ\varphi_0(\theta+\omega).
    \end{align*}
    Hence, we have that
    \begin{align*}
        \mu_4(\theta+2\omega) \leq \pi_1\circ F_{4}^2(x,\theta) = h_4(h_{4}(x) - b g(\theta)) - b g(\theta+\omega) \leq \varphi_1(\theta+2\omega).
    \end{align*}
    Therefore, since $\{\varphi_n\}_n$ is monotone decreasing we have that $F_{4}^2(x,\theta)\in A_{+}$. The case of the other systems is proved using similar arguments. 
\end{proof}
Finally, we obtain a compact invariant set for the piecewise-linear systems. 
\begin{thm}\label{thm_compactinvregion}
    Let $a$ satisfy Assumption \eqref{assumptions_a} and $b\leq b^*(a)$. 
    \begin{itemize}
        \item For the case $i=1,2,3$, the set defined by
        \begin{align*}
            \Lambda_i:= \bigcap_{n\geq0}F^n_{i}(A_{+}),
        \end{align*}
        is a compact invariant set for the map $F_{i}$.
        \item For the case $i=4$, the set
        \begin{align*}
            \tilde{\Lambda}_4:= \bigcap_{n\geq0}F^{2n}_{4}(A_{+})
        \end{align*}
        is a compact invariant set for the map $F^2_{4}$ and the set
        \begin{align*}
            \Lambda_4:= \tilde{\Lambda}_4 \bigcup F_4\left(\tilde{\Lambda}_4\right).
        \end{align*}
        is a compact invariant set for the map $F_{4}$.
    \end{itemize}
\end{thm}

In order to study the upper semicontinuous curve $\varphi_\infty$, we need to study the sequence $\{\varphi_n\}_n$, or equivalently, the sequence $\{\lambda_n\}_n$. A simple computation gives us the following expressions for the $\lambda_n$ that will be useful later.
\begin{itemize}
    \item For the system \eqref{dyn_sys},
    \begin{align*}
        \lambda_{n+1}(\theta) =
        \begin{cases}
            \frac{\pi}{2} - a\mu_1(\theta-\omega) &\text{ if }\lambda_n(\theta-\omega) + \mu_1(\theta-\omega)\geq\frac{\pi}{2a},\\
             a\lambda_n(\theta-\omega) &\text{ if }\lambda_n(\theta-\omega) + \mu_1(\theta-\omega)\leq\frac{\pi}{2a}.
        \end{cases}
    \end{align*}
    \item For the system \eqref{dyn_sys_2},
    \begin{align*}
        \lambda_{n+1}(\theta) =
            \begin{cases}
                a(\mu_2(\theta-\omega)-\delta) &\text{ if }\mu_2(\theta-\omega)-\lambda_n(\theta-\omega)\leq \delta,\\
                a\lambda_n(\theta-\omega) &\text{ if }\mu_2(\theta-\omega)-\lambda_n(\theta-\omega)>\delta.
            \end{cases}
    \end{align*}
    \item For the system \eqref{dyn_sys_3},
    \begin{align*}
        \lambda_{n+1}(\theta) =
        \begin{cases}
            a\mu_3(\theta-\omega)+1 &\text{ if }\mu_3(\theta-\omega)-\lambda_n(\theta-\omega)\leq -\frac{1}{a},\\
            a\lambda_n(\theta-\omega) &\text{ if }\mu_3(\theta-\omega)-\lambda_n(\theta-\omega)\geq -\frac{1}{a}.
        \end{cases}
    \end{align*}
    \item For the system \eqref{dyn_sys_4},
    \begin{align}\label{expression_lambda_4}
        \lambda_{n+1}(\theta) =
        \begin{cases}
            1 - a\mu_4(\theta-\omega) &\text{ if }a[\mu_4(\theta-2\omega)+\lambda_n(\theta-2\omega)] - b g(\theta-2\omega)\leq\frac{1}{a},\\
            a^2\lambda_n(\theta-2\omega) &\text{ if }a[\mu_4(\theta-2\omega)+\lambda_n(\theta-2\omega)] - b g(\theta-2\omega)>\frac{1}{a}.
        \end{cases}
    \end{align}
\end{itemize}
For the system \eqref{dyn_sys_4}, we define the following sets
\begin{align*}
    I_n:=\left\{\theta\in\T\mid a[\mu_4(\theta-2\omega)+\lambda_n(\theta-2\omega)] - b g(\theta-2\omega)\leq\frac{1}{a}\right\}.
\end{align*}
Since $\{\lambda_n\}_n$ is decreasing, if $\theta\in I_{n+1}$ we have that
\begin{align*}
    a[\mu_4(\theta-2\omega)+\lambda_n(\theta-2\omega)] - b g(\theta-2\omega) \leq a[\mu_4(\theta-2\omega)+\lambda_{n+1}(\theta-2\omega)] - b g(\theta-2\omega) \leq\frac{1}{a}.
\end{align*}
Hence, $I_{n+1}\subset I_{n}$. Therefore, the set $I=\bigcap_{n\geq0}I_n$ is a compact set such that for all $\theta\in I$
\begin{align*}
    \lambda_\infty(\theta) &= 1 - a\mu_4(\theta-\omega).
\end{align*}
For the other systems we can obtain a similar result.
\begin{prop}
    Let $a$ satisfy Assumption \eqref{assumptions_a} and $b\leq b^*(a)$.
    \begin{itemize}
        \item For the system \eqref{dyn_sys}, there exists a set $I\subset\T$ such that for all $\theta\in I$
        \begin{align*}
            \lambda_\infty(\theta) &= \frac{\pi}{2} - a\mu_1(\theta-\omega).
        \end{align*}
        \item For the system \eqref{dyn_sys_2}, there exists a set $I\subset\T$ such that for all $\theta\in I$
        \begin{align*}
            \lambda_\infty(\theta-\omega)&=a(\mu_2(\theta-\omega)-\delta).
        \end{align*}
        \item For the system \eqref{dyn_sys_3}, there exists a set $I\subset\T$ such that for all $\theta\in I$
        \begin{align*}
             \lambda_\infty(\theta-\omega)&=a\mu_3(\theta-\omega)+1.
        \end{align*}
        \item For the system \eqref{dyn_sys_4}, there exists a set $I\subset\T$ such that for all $\theta\in I$
        \begin{align*}
            \lambda_\infty(\theta) &= 1 - a\mu_4(\theta-\omega).
        \end{align*}
    \end{itemize}
\end{prop}

\section{Main results}\label{main_sect}
The main objective of this work is to study the four piecewise-linear systems presented in Section \ref{dyn_sys_sect}, with Assumption \ref{assumptions_a}. In this section, we state the main results that describe the bifurcations of the four piecewise-linear systems. We recall that:
\begin{itemize}
    \item For the systems \eqref{dyn_sys}, \eqref{dyn_sys_2} and \eqref{dyn_sys_3}, $\varphi_\infty(\theta)=\lim_n\varphi_n(\theta)$ is an invariant upper semicontinuous curve and $\mu_i$ is an invariant continuous curve, for $i=1,2,3$.
    \item For the system \eqref{dyn_sys_4}, $\varphi_\infty(\theta)=\lim_n\varphi_n(\theta)$ is a two-periodic upper semicontinuous curve and $\mu_4$ is an invariant continuous curve.
\end{itemize}

Our starting point is the formulation of two technical conditions. Both \ref{condition_analytic} and \ref{condition_factorise} are needed in Section \ref{section_b*}.
\begin{enumerate}[label=(\Alph*)]
    \item\label{condition_analytic} The function $g$ is analytic.

    \item\label{condition_factorise} For every $\theta_0\in Z_{\lambda_0}=\{\theta\in\T\mid\lambda_0(\theta)=0\}$, there exists an $\varepsilon>0$ and an even integer $m>0$ such that, for all $\theta\in (\theta_0-\varepsilon,\theta_0+\varepsilon)$, we have that $\lambda_0(\theta)=(\theta-\theta_0)^m q(\theta)$, where $q(\theta)>0$ for all $\theta\in (\theta_0-\varepsilon,\theta_0+\varepsilon)$.
\end{enumerate}
\begin{remark}
    Note that the condition \ref{condition_analytic} implies the condition \ref{condition_factorise} by a classical factorization Theorem of analytic functions \cite[Corollary 3.9]{Conway1978FunctComp}.
\end{remark}
The study of the bifurcations of the four piecewise-linear systems can be synthesised in the following four results. First, we consider the case $0<b<b^*(a)$.

\begin{prop}\label{prop_onlyonecurve}
    Let $a$ satisfy Assumption \eqref{assumptions_a} and $0<b<b^*(a)$. 
    \begin{itemize}
        \item For the system \eqref{dyn_sys} and \eqref{dyn_sys_3}, there exists a unique continuous attracting invariant curve $\varphi_\infty$ such that $\mu_i(\theta)<\varphi_\infty(\theta)\leq\varphi_0(\theta)$ for all $\theta\in\T$ and $i=1,3$.
        \item For the system \eqref{dyn_sys_2}, there exists a unique continuous attracting invariant curve $\varphi_\infty$ such that $\varphi_0(\theta)\leq\varphi_\infty(\theta)<\mu_2(\theta)$ for all $\theta\in\T$.
        \item For the system \eqref{dyn_sys_4}, there exists a unique continuous two-periodic curve $\varphi_\infty$ such that $\mu(\theta)<\varphi_\infty(\theta)\leq\varphi_0(\theta)$ for all $\theta\in\T$. Moreover, $\varphi_\infty$ is attracting as an invariant curve of $F_4^2$.
    \end{itemize} 
\end{prop}
\begin{corollary}\label{coro_numbercurves}
    Let $a$ satisfy Assumption \eqref{assumptions_a} and $0<b<b^*(a)$. Then,
    \begin{itemize}
        \item For the system \eqref{dyn_sys}, there are two unique continuous attracting invariant curves and one unique continuous repelling invariant curve.
        \item For the system \eqref{dyn_sys_2}, there is a unique continuous attracting invariant curve and two unique continuous repelling invariant curves.
        \item For the system \eqref{dyn_sys_3}, there is a unique continuous attracting invariant curve and a unique continuous repelling invariant curves.
        \item For the system \eqref{dyn_sys_4}, are two unique continuous attracting two-periodic curves and one unique continuous repelling invariant curve. Moreover, $\varphi_\infty$ is attracting as an invariant curve of $F_4^2$.
    \end{itemize}
\end{corollary}
\begin{proof}
    In the system \eqref{dyn_sys_4}, by the definition of two-periodic curve, the curve $\mathcal{F}_4(\varphi_\infty)$ is two-periodic and is also attracting. This implies that there are no more two-periodic or invariant curves. Indeed, if we consider a point $(x,\theta)$ that does not belong to the repelling invariant curve, then its iterates tend to one of the two-periodic curves. In the other cases, a similar argument shows that there are only three or two invariant curves.
\end{proof}
Recall that $\lambda_\infty(\theta)=0$ is equivalent to $\varphi_\infty(\theta)=\mu_i(\theta)$ with $i=1,2,3,4$. The following statement establishes some properties of the set of angles where $\varphi_\infty$ intersects with $\mu_i$, for $i=1,2,3,4$.
\begin{thm}\label{big_thm}
    Let $a$ satisfy Assumption \eqref{assumptions_a} and $b=b^*(a)$. Then, for the systems \eqref{dyn_sys}, \eqref{dyn_sys_2}, \eqref{dyn_sys_3} and \eqref{dyn_sys_4}, the set
    \begin{align*}
        A = \{\theta\in\T\mid\lambda_\infty(\theta)=0\} 
    \end{align*}
    is residual and has zero Lebesgue measure. Additionally, $A$ contains a residual set $R$ which is positively and negatively invariant.
\end{thm}
The following result shows that, after the parameter $b^*(a)$, there is a change in the number of invariant curves on each system.
\begin{prop}\label{curve_between}
    Let $a$ satisfy Assumption \eqref{assumptions_a} and $b>b^*(a)$. Then, 
    \begin{itemize}
        \item For the systems \eqref{dyn_sys}, \eqref{dyn_sys_2} and \eqref{dyn_sys_4}, there exists a unique continuous invariant curve.
        \item For the system \eqref{dyn_sys_3}, there exists no continuous invariant curve.
    \end{itemize}
\end{prop}

At $b=b^*(a)$, the closure of the attracting set contains a repelling set. In contrast, we are able to prove that almost every orbit falls into the semicontinuous attracting curve in a finite number of iterates. We denote by $(x_n,\theta_n)$ the $n$-th iteration of a pair $(x_0,\theta_0)\in\R\times\T$. Let $\pi_2\colon\R\times\T\to\T$ be the projection map to the second component.
\begin{thm}\label{finite_num_iter}
    Let $a$ satisfy Assumption in \eqref{assumptions_a} and $b=b^*(a)$. There exists a set $\Omega\subset\R\times\T$ such that $\pi_2(\Omega)$ has full Lebesgue measure and, for any pair $(x_0,\theta_0)\in \Omega$, there exists an $n_0 = n_0(x_0,\theta_0)$ such that $(x_{n_0},\theta_{n_0})$ belongs to the graph of the semicontinuous attracting curve. Moreover, we have that:
    \begin{itemize}
        \item For each system \eqref{dyn_sys}, \eqref{dyn_sys_3} and \eqref{dyn_sys_4}, there exists a set of full Lebesgue measure $E\subset\T$ such that $\Omega=\{(x,\theta)\in\R\times E\mid\mu_i(\theta)<x\}$ with $i=1,3,4$ and every orbit falls into the graph of $\varphi_\infty$ in a finite number of iterates.
        \item For the system \eqref{dyn_sys_2}, there exists a set of full Lebesgue measure $E\subset\T$ such that $\Omega=\{(x,\theta)\in\R\times E\mid\mu_2(\theta)>x>\varphi_\infty(\theta)\}$ and every orbit falls into the graph of $\varphi_\infty$ in a finite number of iterates.
    \end{itemize}
\end{thm}

Furthermore, for $b=b^*(a)$, we are able to compute the Lyapunov exponent of the attracting invariant and two-periodic curves.
\begin{prop}\label{lyap_exp}
    Let $a$ satisfy Assumption in \eqref{assumptions_a} and $b=b^*(a)$. Then the Lyapunov exponent of $\varphi_\infty$ is $-\infty$.
\end{prop}

We can prove that, before and after the parameter $b=b^*(a)$, the curves are piecewise-differentiable.
\begin{thm}\label{regularity_thm}
    Let $a$ satisfy Assumption in \eqref{assumptions_a}. Assume that $0<b<b^*(a)$ or $b>b^*(a)$. Then the curve $\varphi_\infty$ is piecewise $C^{1+\tau}$ with $\tau>0$.
\end{thm}

It is natural to ask whether, for $b=b^*(a)$, the closure of the attracting invariant or two-periodic curve is dense in the region between itself and the repelling one. Recall the definition of the sets $\Lambda_i$ for $i=1,\dots,4$, in Theorem \ref{thm_compactinvregion}.
The next result answers positively to this question.
\begin{thm}\label{densegraph}
    Let $a$ satisfy Assumption  \eqref{assumptions_a} and $b=b^*(a)$.
    \begin{itemize}
        \item For $i=1,2,3$, the set $\Lambda_i$ has positive Lebesgue measure, it can not contain any non-empty open set and the closure of the semicontinuous invariant curve is equal to $\Lambda_i$.
        \item For $i=4$, the set $\Lambda_4$ has positive Lebesgue measure, it can not contain any non-empty open set and the union of the closure of  the two-periodic curves $\varphi_\infty$ and $\mathcal{F}_4(\varphi_\infty)$ is equal to $\Lambda_4$.
    \end{itemize}
\end{thm}

Numerical computations of these systems suggest that, as $b\to b^*(a)$, the attracting invariant curves exhibit fractalization, see Section \ref{fractalization_section} for a precise definition. We write $L(\varphi)$ for the lowest Lipschitz constant of a Lipschitz curve $\varphi$, as detailed in Section \ref{fractalization_section}. If we take a positive $g$, then the fractalization of the invariant or two-periodic curves can be proved.
\begin{thm}\label{weak_frac_thm}
    Let $a$ satisfy Assumption \eqref{assumptions_a}. If $g(\theta)\geq0$ for all $\theta\in\T$, then, for the systems \eqref{dyn_sys}, \eqref{dyn_sys_2}, \eqref{dyn_sys_3} and \eqref{dyn_sys_4}, the family of curves $\{\varphi_\infty(b)\}_b$ fractalize as $b\to b^*(a)$, i.e.,
        \begin{align*}
            \limsup_{b\to b^*(a)} L(\varphi_\infty(b)|_I) = +\infty,
        \end{align*}
        for every interval $I\subset\T$.
\end{thm}

Under certain assumptions, the following result characterises when an invariant curve of a quasiperiodically forced map is noncontinuous.
\begin{thm}\label{mon_characterization}
    Consider a quasiperiodically forced map defined by
    \begin{align*}
        \begin{cases}
            \Bar{x} = h(x) - b g(\theta),\\
            \Bar{\theta} = \theta + \omega \mod2\pi,
        \end{cases}
    \end{align*}
    where $g\colon\T\to\R$ is continuous, $h\colon\R\to\R$ is continuous and monotone increasing and $\omega\not\in2\pi\Q$. 
    Assume that there exists a continuous curve $f_0\colon\T\to\R$ such that:
    \begin{itemize}
        \item Either $f_\infty(\theta) \leq f_0(\theta)$ or $f_\infty(\theta)\geq f_0(\theta)$ for all $\theta\in\T$.
        \item The sequence defined by the fixed point map of the system $$f_n(\theta):=h\circ f_{n-1}(\theta-\omega)- b g(\theta-\omega)$$ has a pointwise limit $f_\infty$.
    \end{itemize}
    Define the following sequence of curves:
    \begin{align*}
        f_m^{\downarrow}(\theta)&:=\inf_{k\leq m}f_k(\theta),\\
        f_m^{\uparrow}(\theta)&:=\sup_{k\leq m}f_k(\theta).
    \end{align*}
    
    Then $f_\infty$ is invariant and $f_\infty$ is continuous if and only if either $\{f_m^{\downarrow}\}_m$ or $\{f_m^{\uparrow}\}_m$ is uniformly convergent. In particular,
    \begin{itemize}
        \item If $f_\infty\leq f_0$, then $f_\infty$ is continuous if and only if $\{f_m^{\downarrow}\}_m$ is uniformly convergent.
        \item If $f_\infty\geq f_0$, then $f_\infty$ is continuous if and only if $\{f_m^{\uparrow}\}_m$ is uniformly convergent.
    \end{itemize}
\end{thm}

As a consequence, we obtain the following result for our piecewise-linear systems.
\begin{corollary}\label{char_corollary}
    Let $a$ satisfy Assumption \eqref{assumptions_a} and $b=b^*(a)$.   
    Then, for the systems \eqref{dyn_sys}, \eqref{dyn_sys_2}, \eqref{dyn_sys_3} and \eqref{dyn_sys_4}, the sequence $\{\varphi_n\}_n$ is not uniformly convergent and satisfies that for every interval $I\subset\T$
        \begin{align*}
            \limsup_n L(\varphi_n|_I)=+\infty.
        \end{align*}
\end{corollary}

Finally, we can prove that these systems have only an attractor when $\abs{a}<1$.
\begin{thm}\label{thm_uniform_contraction_case}
    Let $\abs{a}<1$ and $b\in\R$. Then $F_i$ has a unique Lipschitz invariant curve for each $i\in\{1,2,3,4\}$.
\end{thm}

\section{Proofs of the main results and additional statements}\label{proof_sect}
Each $\mu_i$ plays a very similar role on the proofs for each system. Therefore, unless there is a possible confusion, we will denote by $\mu$ any of the invariant curves $\mu_i$.
\subsection{A noncontinuous attractor}\label{section_b*}
In this section, we assume that $b=b^*(a)$.
The main objective in this section is to prove that the curve $\varphi_\infty$ is not continuous. This is done by proving that the set of angles where $\lambda_\infty$ is zero is a residual set of zero measure. For all the systems, we need to control the absolute extreme points of the invariant curve $\mu$. From Lemma \ref{lemma_distance}, the set $Z_{\lambda_0}=\{\theta\in\T\mid\lambda_0(\theta)=0\}$ is a subset of the set of absolute extreme points of $\mu$. We state two conditions (\ref{condition_analytic} and \ref{condition_factorise}) that grant that $Z_{\lambda_0}$ is a finite set of points.
\begin{itemize}
    \item Assuming \ref{condition_analytic}, we have that $\lambda_0$ is analytic, see Remark \ref{remarks_curves}\ref{regularity_remark}. Then the zero set $Z_{\lambda_0}$ has no accumulation points. Since for every $\theta\in\T$ we have that $\lambda_0(\theta)\geq\lambda_\infty(\theta)\geq0$, every zero has a finite even order. Hence, since $\T$ is compact, the set $Z_{\lambda_0}$ is equal to a finite set of points. 

    \item Assuming \ref{condition_factorise}, we have that $Z_{\lambda_0}$ has no accumulation points. Hence, as $\T$ is compact, the set $Z_{\lambda_0}$ is equal to a finite set of points. 
\end{itemize}
In any of these two cases, $Z_{\lambda_0}$ has zero measure. Therefore, for all $n$ the functions $\log\circ\lambda_n$ are integrable. The integrability of $\log\circ\lambda_n$ is the key to obtain the non-continuity of the semicontinuous function $\varphi_\infty$. We give two interesting examples for the system \eqref{dyn_sys}:
\begin{enumerate}
    \item\label{example_finite} We require condition \ref{condition_factorise}. Given a finite sequence $\{\theta_n\}_{n=0}^N\subset\T$. Let $\boldsymbol\mu_1$ be in $C^\infty(\T)$ such that $\boldsymbol\mu_1(\theta_n)=1$ for $n=0,\dots,N$, and that $\boldsymbol\mu_1(\theta)<1$ if $\theta\not=\theta_n$ for $n=0,\dots,N$. With this $\mu_1$ we have a finite set of points such that for a value of the parameter $b$ we have that $\mu_1(\theta_n)=\frac{\pi}{2a}$ for $n=0,\dots,N$.
    \item\label{example_pos_meas} Given an $\varepsilon>0$ small enough and an interior point $\theta_0\in\T$, consider the open interval $I_\varepsilon = (\theta_0,\theta_0+\varepsilon)$ of $\T$ of length $\varepsilon$.  Let $\boldsymbol\mu_1$ be in $C^\infty(\T)$ such that $\boldsymbol\mu_1(\theta)=1$ for any angle $\theta\in I_\varepsilon$, and $\boldsymbol\mu_1(\theta)<1$ for any angle $\theta\in \T\smallsetminus I_\varepsilon$. With this $\mu_1$ we have an open set of points such that for a value of the parameter $b$ we will have that $\mu_1(\theta)=\frac{\pi}{2a}$ for any angle $\theta\in I_\varepsilon$.
\end{enumerate}
Note that these examples can be reproduced in the other systems changing minor details of the functions.
Observe that, in the example \ref{example_finite} with $b=b^*(a)$, we have that $Z_{\lambda_0}$ has zero measure. In the example \ref{example_pos_meas}, with $b=b^*(a)$, $Z_{\lambda_0}$ has measure $\varepsilon$. In this case, for $\theta\in I_\varepsilon$, we have that $\lambda_0(\theta)=0$. Hence, for $\theta\in I_\varepsilon$, we have that $\lambda_\infty(\theta)=0$. Using the invariance and the compactness of $\T$, we obtain that $\varphi_\infty(\theta)=\mu(\theta)$ for all $\theta\in\T$. 

This situation can be even more extreme. For example, assume that $\boldsymbol\mu_1(\theta)=1$ for a residual set of small positive measure $I_0$, and for the rest of angles $\boldsymbol\mu_1(\theta)<1$. Therefore, for a value of the parameter $b$, the curves $\varphi_\infty$ and $\mu_1$ coincide on $I_0$. The set $I_0$ can be constructed from the complementary of a fat Cantor set. 

From this reasoning, to ensure that $\log\circ\lambda_n$ is integrable, we need to guarantee that one of the conditions \ref{condition_analytic} or \ref{condition_factorise} is satisfied.
For simplicity of the argument, we assume that $\lambda_0$ has a unique zero. After the following lemma, we will see that indeed having more than one zero of $\lambda_0$ does not affect the argument.

\begin{lemma}\label{zeros_lambda}
    Let $a$ satisfy Assumption \eqref{assumptions_a} and $b = b^*(a)$. The curve $\lambda_n$ has a finite set of of zeros. Concretely, let $\theta_0$ be the only zero of $\lambda_0$. Then we have the following:
    \begin{itemize}
        \item For the systems \eqref{dyn_sys}, \eqref{dyn_sys_2} and \eqref{dyn_sys_3}, $\lambda_n$ has exactly $n+1$ zeros $\theta_0,\theta_0+\omega,\dots,\theta_0+n\omega$.
        \item For the system \eqref{dyn_sys_4}, $\lambda_n$ has exactly $n+1$ zeros $\theta_0,\theta_0+2\omega,\dots,\theta_0+2n\omega$.
    \end{itemize}
\end{lemma}
\begin{proof}
    Consider the system \eqref{dyn_sys_4}.
    Given $n\geq 1$, we see that
    \begin{align*}
        F_{4}^{2(n-1)}(\mu(\theta_0-2n\omega),\theta_0-2n\omega) &= (\mu(\theta_0-2\omega),\theta_0-2\omega),\\
        F_{4}^{2(n-1)}(\varphi_0(\theta_0-2n\omega),\theta_0-2n\omega) &= (\varphi_{n-1}(\theta_0-2\omega),\theta_0-2\omega).
    \end{align*}
    Since $\mu(\theta_0-\omega)=\frac{1}{a}$ and $\mu(\theta_0-2\omega)\leq \varphi_{n-1}(\theta_0-2\omega)$,
    \begin{align*}
        h_4(\varphi_{n-1}(\theta_0-2\omega))-b g(\theta)\leq h_4(\mu(\theta_0-2\omega))-b g(\theta) = \mu(\theta_0-\omega) \leq \frac{1}{a}.
    \end{align*}
    Therefore, we obtain that
    \begin{align*}
        F_{4}^2(\mu(\theta_0-2\omega),\theta_0-2\omega) = F_{4}^2(\varphi_{n-1}(\theta_0-2\omega),\theta_0-2\omega).
    \end{align*}
    As a consequence, we obtain that $\lambda_n(\theta_0)=\varphi_n(\theta_0)-\mu(\theta_0)=0$. If $\theta$ is a zero of $\lambda_n$, from \eqref{expression_lambda_4} we have that
    \begin{align*}
        \lambda_{n+1}(\theta+2\omega) = 
        \begin{cases}
            1 - a\mu(\theta+\omega) &\text{ if }a\mu(\theta) - b g(\theta)\leq\frac{1}{a},\\
            a^2\lambda_n(\theta) &\text{ if }a\mu(\theta) - b g(\theta)>\frac{1}{a}.
        \end{cases}
    \end{align*}
    Since $a\mu(\theta) - b g(\theta) = \mu(\theta+\omega)\geq\frac{1}{a}$, $\lambda_{n+1}(\theta+2\omega)=a^2\lambda_n(\theta)$. Then $\theta+2\omega$ is a zero of $\lambda_{n+1}$. 
    If we consider the case where $\theta$ is the minimum of $\mu$, then $\theta=\theta_0-\omega$ so
    \begin{align*}
        \lambda_{n+1}(\theta+\omega) = 
        \begin{cases}
            1 - a\mu(\theta) &\text{ if }a\mu(\theta-\omega) - b g(\theta-\omega)\leq\frac{1}{a},\\
            a^2\lambda_n(\theta-\omega) &\text{ if }a\mu(\theta-\omega) - b g(\theta-\omega)>\frac{1}{a}.
        \end{cases}
    \end{align*}
    Since $a\mu(\theta-\omega) - b g(\theta-\omega) = \mu(\theta) = \frac{1}{a}$, $\lambda_{n+1}(\theta+\omega) = 1 - a\mu(\theta) = 0$.
    
    The uniqueness follows from induction. For $n=0$ this is true, so we assume that $\lambda_{n+1}$ has an extra zero $\theta_1$. From \eqref{expression_lambda_4} there are two cases,
    \begin{enumerate}
        \item $1 - a\mu(\theta+\omega)=0$.
        We have that $\mu(\theta_1-\omega)=\frac{1}{a}$ so $\theta_1=\theta_0+2\omega$.
        
        \item $a^2\lambda_n(\theta)=0$.
        In this case $\lambda_n$ has an extra zero, but this is contradictory with the induction hypothesis which is that $\lambda_{n}$ has no extra zero.
    \end{enumerate}
    A similar argument is valid for the systems \eqref{dyn_sys}, \eqref{dyn_sys_2} and \eqref{dyn_sys_3}. For these systems you have to consider the one-iteration map instead of the two-iteration map.
\end{proof}
We obtained that $Z=\{\theta\in\T\mid\lambda_n(\theta)=0\text{ for any }n\geq0\}$ has zero measure. We will use this to ensure the integrability of $\log\circ\lambda_n$. The finite union of zero measure sets is of zero measure. Therefore, having more than one zero of $\lambda_0$ does not affect the integrability of $\log\circ\lambda_n$, for all $n\geq0$.
Recall that $b=b^*(a)$ and $\theta_0$ is the unique angle such that $\lambda_0(\theta_0)=0$. When $\lambda_n$ is small, the curve $\varphi_n$ is near the invariant curve $\mu$. Therefore, we have that
\begin{itemize}
    \item For the systems \eqref{dyn_sys}, \eqref{dyn_sys_2} and \eqref{dyn_sys_3}, the sequence $\{\lambda_n\}_n$ verifies that $\lambda_{n+1}(\theta+\omega)=a\lambda_n(\theta)$.
    \item For the system \eqref{dyn_sys_4}, the sequence $\{\lambda_n\}_n$ verifies that $\lambda_{n+1}(\theta+\omega)=a^2\lambda_n(\theta)$.
\end{itemize}
This justifies the following definition.
\begin{definition}\label{definition_psi}
    Let $a$ satisfy Assumption \eqref{assumptions_a}. We define the following sequence of functions $\{\psi_n\}_n$.
    \begin{itemize}
        \item For the systems \eqref{dyn_sys}, \eqref{dyn_sys_2} and \eqref{dyn_sys_3} the sequence is defined by
        \begin{align*}
            \psi_n(\theta):=
            \begin{cases}
                \frac{\lambda_{n+1}(\theta+\omega)}{\lambda_n(\theta)} &\text{ if }\lambda_n(\theta)\not=0,\\
                a &\text{ if }\lambda_n(\theta)=0.
            \end{cases}
        \end{align*}
        \item For the system \eqref{dyn_sys_4} the sequence is defined by
        \begin{align}\label{psi_def}
            \psi_n(\theta):=
            \begin{cases}
                \frac{\lambda_{n+1}(\theta+2\omega)}{\lambda_n(\theta)} &\text{ if }\lambda_n(\theta)\not=0,\\
                a^2 &\text{ if }\lambda_n(\theta)=0.
            \end{cases}
        \end{align}
    \end{itemize}
\end{definition}

We have to study how the functions $\psi_n$ behave. First, we show this sequence is bounded and monotone.
\begin{lemma}\label{integ_psy}
    Let $a$ satisfy Assumption \eqref{assumptions_a}. Then we have that $0\leq \psi_{n-1}(\theta)\leq \psi_{n}(\theta)$, for all $n\geq1$ and for all $\theta\in\T$. Moreover,
    \begin{itemize}
        \item For the systems \eqref{dyn_sys}, \eqref{dyn_sys_2} and \eqref{dyn_sys_3}, we have that $\psi_{n}(\theta)\leq a$, for all $n\geq0$ and for all $\theta\in\T$.
        \item For system \eqref{dyn_sys_4}, we have that $\psi_{n}(\theta)\leq a^2$, for all $n\geq0$ and for all $\theta\in\T$.
    \end{itemize}
\end{lemma}
\begin{proof}
    In case of system \eqref{dyn_sys_4}, we use the expression \eqref{expression_lambda_4} in equation \eqref{psi_def} to obtain
    \begin{align}\label{expr_psy_4}
        \psi_n(\theta)=
        \begin{cases}
            \frac{1 - a\mu(\theta+\omega)}{\lambda_n(\theta)} &\text{ if }a[\mu(\theta)+\lambda_n(\theta)] - b g(\theta)\leq\frac{1}{a},\\
            a^2 &\text{ if }a[\mu(\theta)+\lambda_n(\theta)] - b g(\theta)>\frac{1}{a}.
        \end{cases}
    \end{align}
    So by the expression \eqref{expr_psy_4}, if $a[\mu(\theta)+\lambda_n(\theta)] - b g(\theta)\leq\frac{1}{a}$, then $a\lambda_n(\theta)\leq \frac{1}{a}-\mu(\theta+\omega)$. Hence,
    \begin{align*}
        \psi_n(\theta) = \frac{\lambda_{0}(\theta+2\omega)}{\lambda_n(\theta)} \leq \frac{\lambda_{0}(\theta+2\omega)}{\frac{1}{a^2}-\frac{1}{a}\mu(\theta+\omega)} = a^2\frac{\lambda_{0}(\theta+2\omega)}{\lambda_{0}(\theta+2\omega)} = a^2.
    \end{align*}
    Since $\{\lambda_n\}_n$ is decreasing we have only three cases:
    \begin{enumerate}
        \item $a[\mu(\theta)+\lambda_{n+1}(\theta)] - b g(\theta)\geq a[\mu(\theta)+\lambda_n(\theta)] - b g(\theta)>\frac{1}{a}$. In this case, we have that
            \begin{align*}
                \psi_{n+1}(\theta) &= a^2,\\
                \psi_n(\theta) &= a^2.
            \end{align*}
            Hence, $\psi_{n}(\theta) \leq \psi_{n+1}(\theta)$.
            
        \item $a[\mu(\theta)+\lambda_{n+1}(\theta)] - b g(\theta)>\frac{1}{a}\geq a[\mu(\theta)+\lambda_n(\theta)] - b g(\theta)$. In this case, we have that
            \begin{align*}
                \psi_{n+1}(\theta) &= a^2.
            \end{align*}
            Hence, $\psi_{n}(\theta) \leq \psi_{n+1}(\theta)$.         
        \item $\frac{1}{a}\geq a[\mu(\theta)+\lambda_{n+1}(\theta)] - b g(\theta)\geq a[\mu(\theta)+\lambda_n(\theta)] - b g(\theta)$. In this case, we have that
            \begin{align*}
                \psi_{n+1}(\theta) &= \frac{\lambda_{0}(\theta+2\omega)}{\lambda_{n+1}(\theta)},\\
                \psi_n(\theta) &= \frac{\lambda_{0}(\theta+2\omega)}{\lambda_n(\theta)}.
            \end{align*}
            Since $\{\lambda_n\}_n$ is decreasing, we get $\psi_{n}(\theta) \leq \psi_{n+1}(\theta)$.
    \end{enumerate}
    Therefore, $0\leq \psi_{n}(\theta) \leq \psi_{n+1}(\theta)\leq a^2$, for all $n\geq0$ and for all $\theta\in\T$. A similar argument can be done to obtain the result for the other systems.
\end{proof}
We proceed in proving the pointwise convergence of the sequence $\{\psi_n\}_n$ and the integrability of its pointwise limit.
\begin{lemma}\label{lem_int_psi}
    Let $a$ satisfy Assumption \eqref{assumptions_a}. Then we have that the sequence $\{\psi_n\}_n$ converges pointwise to an integrable function $\psi$. Furthermore,
    \begin{align*}
        \int_\T \log(\psi(\theta))d\theta \leq 0.
    \end{align*}
\end{lemma}
\begin{proof}
    By the previous Lemma \ref{integ_psy}, in order to show that $\log(\psi_n)$ are integrable, we only need to show that $\log(\psi_0)$ is integrable. For every system, we can determine the expression of $\psi_0$ in terms of $\lambda_0$.
    Therefore, $\psi_0$ is well defined, continuous and it only has one zero of even multiplicity in $\theta=\theta_0$. In the general case, \ref{condition_analytic} or \ref{condition_factorise} imply that we can only have a finite number of zeros with even multiplicity. Consequently, $\log(\psi_0)$ is integrable. Moreover, in case of system \eqref{dyn_sys_4}, since $\{\lambda_{n}\}_n$ is decreasing, we have that
    \begin{align}\label{ineq_psi}
        \psi_n(\theta)\leq \frac{\lambda_{n}(\theta+2\omega)}{\lambda_n(\theta)},
    \end{align}
    for a set of $\theta$ of total measure (those which $\lambda_n(\theta)\not=0$). Hence, for these values of $\theta$ we have that
    \begin{align*}
        \log(\psi_n(\theta))\leq \log(\lambda_{n}(\theta+2\omega))-\log(\lambda_n(\theta)).
    \end{align*}
    Since $\lambda_n$ is an integrable function (its zeros are of even multiplicity),
    \begin{align*}
        \int_\T\log(\psi_n(\theta))\leq \int_\T\log(\lambda_{n}(\theta+2\omega))-\int_\T\log(\lambda_n(\theta))=0.
    \end{align*}
    Applying the Dominated Convergence Theorem we obtain that
    \begin{align*}
        \int_\T\log(\psi(\theta)) = \lim_n \int_\T\log(\psi_n(\theta)) \leq 0.
    \end{align*}
    The result for the systems \eqref{dyn_sys}, \eqref{dyn_sys_2} and \eqref{dyn_sys_3} can be deduced from replacing $\lambda_{n}(\theta+2\omega)$ for $\lambda_{n}(\theta+\omega)$ in equation \eqref{ineq_psi} and repeating the same argument.
\end{proof}
It is convenient to recall the sequences defined for each system in Lemma \ref{montonicity_lemma}. We are ready to prove Theorem \ref{big_thm}.
\begin{proof}[Proof of Theorem \ref{big_thm}]
    Consider the system \eqref{dyn_sys_4}. Since $\theta_0\in A$, $\theta_0+k\omega\in A$ for all $k\in\N$. Since the angle $\frac{\omega}{2\pi}$ is irrational, $A$ is dense in $\T$. Since $\varphi_\infty$ is upper semicontinuous, by the semicontinuity Lemma (see \cite[Theorem 11.4]{roo1982real}), the set of continuity $R_0\subset\T$ of $\varphi_\infty$ is a residual set.         
    We can show that actually $R_0=A$. 
    First assume that $\theta_*\in R_0$. Since $A$ is dense, we can consider a sequence $\{\theta_n\}_n\subset A$ such that $\theta_n\to\theta_*$. Since $\theta_*\in R_0$, we have that $\varphi_\infty(\theta_n)\to\varphi_\infty(\theta_*)$. Since $\mu$ is continuous, we have that $\mu(\theta_n)\to\mu(\theta_*)$. Therefore, as the sequence $\{\theta_n\}_n$ is in $A$, we have that $\varphi_\infty(\theta_n)=\mu(\theta_n)$. Hence, $\varphi_\infty(\theta_*)=\mu(\theta_*)$. Then we have the inclusion $R_0\subset A$. In the other direction, assume that $\theta_*\in A$. We need to prove that for any sequence $\{\theta_n\}_n\subset A$ such that $\theta_n\to\theta_*$ we have that $\varphi_\infty(\theta_n)\to\varphi_\infty(\theta_*)$. Consider the following inequality given by the upper semicontinuity of $\varphi_\infty$:
    \begin{align*}
        \mu(\theta_*) = \liminf_n \mu(\theta_n) = \liminf_n \varphi_\infty(\theta_n) \leq \limsup_n \varphi_\infty(\theta_n) \leq \varphi_\infty(\theta_*) = \mu(\theta_*).
    \end{align*}
    Hence, we have proved that $A\subset R_0$. We conclude that $A=R_0$. Let $j\in\Z$. We can repeat the argument for $\lambda_n(\theta+j\omega)$. Hence, we have the residual set $R_j$ of points of continuity. It is clear that $R_j=\{\theta-j\omega\mid\theta\in R_0\}$. Therefore, the residual set $R=\bigcap_{j\in\Z}R_j$ is positively and negatively invariant.
    In order to show that $A$ has zero measure, first we prove that
    \begin{align*}
        A\smallsetminus\{\theta_0-\omega\}\subset\{\theta_0\in \T\mid\psi(\theta)=a^2\}.
    \end{align*}
    If we take a $\theta\in A\smallsetminus\{\theta_0-\omega\}$, then $\varphi_\infty(\theta)=\mu(\theta)>\frac{1}{a}$. Hence, there exists an $n_0$ such that, for all $n\geq n_0$, we have that $\mu(\theta)+\lambda_n(\theta)=\varphi_n(\theta)>\frac{1}{a}$. This implies that, for all $n\geq n_0$, we have that $\lambda_{n+1}(\theta+\omega)=a^2\lambda_{n}(\theta)$. Hence, $\psi_n(\theta)=a^2$ for all $n\geq n_0$. Assume that $A$ has positive measure. By ergodicity, $A$ has total measure. This implies that the set of $\theta\in\T$ such that $\psi(\theta)=a^2$ has total measure. Since $a^2>1$, we obtain
        \begin{align*}
            \int_\T \log(\psi(\theta)) d\theta = \log(a^2) > 0.
        \end{align*}
    This is contradictory with Lemma \ref{lem_int_psi}. Hence, $A$ has zero measure. The proof for the systems \eqref{dyn_sys}, \eqref{dyn_sys_2} and \eqref{dyn_sys_3} is obtained from replacing $\frac{1}{a}$ for the corresponding value where the definition of $h_i$ changes. Additionally, it is needed to replace $a^2$ by $a$.
\end{proof}

The derivative of the function $h_i$ is well defined except in a point $x_0$ (for the system \eqref{dyn_sys_4} the derivative of $h_4$ is not well defined at $x_0=\frac{1}{a}$). We define $h_i'(x_0)=a$.
We can proceed to compute the Lyapunov exponent of the semicontinuous invariant or two-periodic curve.
\begin{proof}[Proof of Proposition \ref{lyap_exp}]
    Consider the system \eqref{dyn_sys_4}. First we show that the set
    \begin{align*}
        B = \left\{\theta\in\T\mid a\varphi_\infty(\theta)-b g(\theta)<\frac{1}{a}\right\}
    \end{align*}
    has positive measure. If we assume that $B$ has zero measure, then the Fourier coefficients of $\lambda_\infty=\varphi_\infty-\mu$ are
    \begin{align*}
        \lambda^{(k)} = \frac{1}{2\pi}\int_\T \lambda(\theta)e^{-ik\omega}d\theta = \frac{1}{2\pi}\int_{\T\smallsetminus B} \lambda(\theta)e^{-ik\omega}d\theta.
    \end{align*}
    Since $a(\mu(\theta)+\lambda_\infty(\theta))-b g(\theta)=a\varphi_\infty(\theta)-b g(\theta)\geq\frac{1}{a}$ for $\theta\in\T\smallsetminus B$, $\lambda_\infty(\theta+\omega)=a^2\lambda_\infty(\theta)$. Hence, we have that $\lambda^{(k)}e^{ik\omega} = a^2\lambda^{(k)}$. Since $a^2>1$, we obtain that $\lambda^{(k)}=0$ for all $k$. Therefore, $\lambda=0$ for almost every $\theta$. This is contradictory with the second statement of Theorem \ref{big_thm}. Note that, if we let $\Bar{\Bar{x}}=f(x,\theta)$, then $\frac{\partial f}{\partial x}(x,\theta) = h_a'(h_a(x)-b g(\theta))h_{a}'(x)$. Recall that $\varphi_\infty(\theta)\geq\mu(\theta)\geq\frac{1}{a}$ for all $\theta\in\T$, so $\frac{\partial f}{\partial x}(\varphi_\infty(\theta),\theta) = h_a'(a\varphi_\infty(\theta)-b g(\theta))h_{a}'(\varphi_\infty(\theta))$.

    Hence, the Lyapunov exponent is
    \begin{align*}
        \Lambda &= \frac{1}{2\pi}\int_\T\log\left(\frac{\partial f}{\partial x}(\varphi_\infty(\theta),\theta)\right)d\theta\\
        &= \frac{1}{2\pi}\int_{\T\smallsetminus B}\log\left(\frac{\partial f}{\partial x}(\varphi_\infty(\theta),\theta)\right)d\theta + \frac{1}{2\pi}\int_B\log\left(\frac{\partial f}{\partial x}(\varphi_\infty(\theta),\theta)\right)d\theta\\
        &= \frac{1}{2\pi}\log(a^2)m(\T\smallsetminus B) + \frac{1}{2\pi}(-\infty)m(B) = -\infty.
    \end{align*}
    The proof for the systems \eqref{dyn_sys}, \eqref{dyn_sys_2} and \eqref{dyn_sys_3} is obtained, with minor modifications, from replacing $\frac{1}{a}$ for the corresponding value where the definition of $h_i$ changes. Additionally, it is needed to replace $a^2$ by $a$.
\end{proof}
Observe that, from the invariant equation, we deduce that the closure of the attracting set $\Lambda_i$ contains a repelling invariant curve for each system $i=1,2,3,4$.
We proceed in proving Theorem \ref{finite_num_iter} where the attracting character of the curve $\varphi_\infty$ is described.
\begin{proof}[Proof of Theorem \ref{finite_num_iter}]
    Consider the system \eqref{dyn_sys_4}.
    Take an arbitrary point $(x_0,\theta_0)\in\R\times\T$ such that $\mu(\theta_0)<x_0$. Then we have that $\mu(\theta_{2n})<x_{2n}$. Let $T_\omega\colon\T\to\T$ denote the irrational rotation $T_\omega(\theta)=\theta+\omega$. Recall that the sets $A$ and $B$ are defined by
    \begin{align*}
        A &= \{\theta\in\T\mid\mu(\theta)=\varphi_\infty(\theta)\},\\
        B &= \left\{\theta\in\T\mid a\varphi_\infty(\theta)-b g(\theta)<\frac{1}{a}\right\}.
    \end{align*}
    Then we define the set $C$ as
    \begin{align*}
        C = \bigcup_{n\in\N} T_{2\omega}^{-n}(B).
    \end{align*}
    Since the Lebesgue measure is invariant and ergodic by the irrational rotation $T_{2\omega}$, we use \cite[Theorem 1.5]{walters2000introduction} to obtain that $C$ has total measure. We have several cases:
    \begin{enumerate}
        \item Case $\varphi_\infty(\theta_0)< x_0$ for $\theta_0\in C$.

        Since $\theta_0\in C$, there exists an $n_0$ such that $\theta_{n_0}\in B$. Therefore, $\frac{1}{a}>a\varphi_\infty(\theta_{n_0})-b g(\theta_{n_0})>ax_{n_0}-b g(\theta_{n_0})$ and thus $\varphi_\infty(\theta_{n_0+1})= x_{n_0+1}$.

        \item Case $\mu(\theta_0)<x_0<\varphi_\infty(\theta_0)$ for $\theta_0\in A^c$.

        Since $A$ is a set of zero measure, $A^c$ is a set of full measure. Assume $x_{n}>\frac{1}{a}$ for all $n$. Then $x_{2n}-\mu(\theta_{2n}) = a^{2n}(x_0-\mu(\theta_0))$. Since $a<-1$, $x_{2n}-\mu(\theta_{2n})$ is unbounded. This implies that $x_{2n}$ is unbounded, which is impossible. Hence, there exists an $n_0$ such that $x_{n_0}\leq\frac{1}{a}$. Therefore, there exists a $\theta_*\in\T$ such that $x_{n_0+1}=\varphi_0(\theta_*)>\varphi_\infty(\theta_*)$.
    \end{enumerate}
    We define the set $E=C\cap A^c$ which is also of full measure. Hence, we finally define the set $\Omega=\{(x,\theta)\in\R\times E\mid\mu(\theta)<x\}$. The proof for the rest of the systems follows a similar argument. The main differences are that the set $B$ changes depending on each system:
    \begin{itemize}
        \item For the system \eqref{dyn_sys}, $B = \left\{\theta\in\T\mid\varphi_\infty(\theta)>\frac{\pi}{2a}\right\}$.
        \item For the system \eqref{dyn_sys_2}, $B = \left\{\theta\in\T\mid\varphi_\infty(\theta)<\delta\right\}$.
        \item For the system \eqref{dyn_sys_3}, $B = \left\{\theta\in\T\mid\varphi_\infty(\theta)<-\frac{1}{a}\right\}$.
    \end{itemize}
\end{proof}
\begin{remark}
        Note that, for all $(\theta_0,x_0)\in\{(\theta,x)\mid\mu(\theta)\leq x\leq\varphi_\infty(\theta)\}$, there exists an $n_0$ such that $(\theta_{n_0},x_{n_0})$ belongs to the graph of the semicontinuous curve. This implies that, between the curves, the semicontinuous curve attracts all the points for each fiber.
    \end{remark}

We now prove that the graph of the semicontinuous invariant curve is dense in the compact invariant set $\Lambda_i$, see Theorem \ref{thm_compactinvregion}. The curve $\varphi_\infty$ is a measurable function by the semicontinuity. Then, we can apply Fubini's Theorem to the indicator function of the set $\tilde{\Lambda}_4$. We can compute that
\begin{align*}
    \text{meas}(\Lambda_4)=\int_\T\int_{\R}\chi_{\Lambda_4}(\theta,x) dx d\theta = \int_\T\int_{\mu(\theta)}^{\varphi_\infty(\theta)}dx d\theta = \int_\T \lambda_\infty(\theta)d\theta.
\end{align*}
Since we have proved, in Theorem \ref{big_thm}, that the set $\{\theta\in\T\mid \lambda_\infty(\theta)>0\}$ has total measure, we have that $\text{meas}(\Lambda_4)>0$. We can do a similar argument for the rest of systems. Therefore, we obtain that, for $i=1,2,3,4$, the invariant set $\Lambda_i$ has positive Lebesgue measure.
Recall that, given a set $A\subset X$ from a metric space $X$, the distance from a point $x\in X$ to the set $A$ is defined by $d(x,A)=\inf_{a\in A}d(x,a)$. It is known that $d(x,A)=0$ if and only if $x\in\Bar{A}$. We are ready to prove Theorem \ref{densegraph}.
\begin{proof}[Proof of Theorem \ref{densegraph}]
     Consider the system \eqref{dyn_sys_4}. In the case that $\mu(\theta)=\varphi_\infty(\theta)$, for all $\theta\in\T$, then $\lambda_\infty(\theta)=0$, for all $\theta\in\T$. This is impossible since $\lambda_\infty(\theta)>0$ almost everywhere. Hence, there exists $z_0=(\theta_*,x_*)\in\Lambda_4$ such that $\mu(\theta_*)<x_*<\varphi_\infty(\theta_*)$. We define $\Phi=\{(x,\theta)\in\T\times\R\mid\varphi_\infty(\theta)=x\}$ and $M=\{(\theta,x)\in\T\times\R\mid x = \mu(\theta)\}$. Assume, by contradiction, that $d(z_0,\Phi)>0$. Let $\delta=\min\{d(z_0,M),d(z_0,\Phi)\}>0$ and $B_\delta(z_0)$ be an open ball of radius $\delta$ centred at $z_0$. If $z_1=(\theta_1,x_1)\in B_\delta(z_0)$, then $\mu(\theta_1)<x_1<\varphi_\infty(\theta_1)$. Otherwise $\mu(\theta_1)=x_1=\varphi_\infty(\theta_1)$, which is contrary to the definition of $\delta$. Since $b=b^*(a)$, we have that $\{\theta\in\T\mid \mu(\theta)=\varphi_\infty(\theta)\}$ is dense in $\T$. Hence, we can consider a sequence $\{\theta_n\}_n\to\theta_*$ such that $\mu(\theta_n)=\varphi_\infty(\theta_n)$, for every $n\geq0$. Therefore, there is an $N>0$ such that $d(\theta_N,\theta_*)<\delta$ and $\mu(\theta_N)=\varphi_\infty(\theta_N)$. Let $z_N$ be such that $z_N=(\theta_N,x_*)$. We can compute that 
     \begin{align*}
        d(z_n,z_0) = d(\theta_N,\theta_*) < \delta.
     \end{align*}
     Now we have that $\mu(\theta_N)=\varphi_\infty(\theta_N)$. Then $\mu(\theta_N)=x_*=\varphi_\infty(\theta_N)$. The last equality is a contradiction with the fact that, for all $z_1=(\theta_1,x_1)\in B_\delta(z_0)$, we have $\mu(\theta_1)<x_1<\varphi_\infty(\theta_1)$. Hence, we have that $d(z_0,\Phi)=0$. Therefore, $z_0$ is in the closure of $\Phi$. The proof of the rest of systems follows from similar arguments.
\end{proof}
\begin{remark}\label{keller_remark}
    Notice that, in the proof of Theorem \ref{densegraph}, we only use that $\{\theta\in\T\mid\lambda_\infty(\theta)=0\}$ is dense and that there exists a $\theta^*\in\T$ such that $\lambda_\infty(\theta^*)>0$. In \cite{keller1996sna}, we have similar hypothesis. In the work, $\{\theta\in\T\mid\phi(\theta)=0\}$ is dense and $\{\theta\in\T\mid\phi(\theta)>0\}$ has full Lebesgue measure, see \cite[Theorem 1(3)]{keller1996sna}. Therefore, the curve $\phi$ is dense in $\{(\theta,x)\mid0\leq x\leq\phi(\theta)\}$.
\end{remark}

\subsection{Before the nonsmooth bifurcation}\label{three_curves_secion}
We assume that $0<b<b^*(a)$ and $a$ satisfies Assumption \eqref{assumptions_a}. For $0<b<b^*(a)$, we are able to prove the continuity of the curves and we determine the number of continuous invariant or two-periodic curves each system has.
\begin{proof}[Proof of Proposition \ref{prop_onlyonecurve}]
    Consider the system \eqref{dyn_sys_4}. We define $d_b$ as
    \begin{align*}
        d_b = \min_{\theta\in\T}\left(\frac{1}{a}-\mu(\theta)\right)<0.
    \end{align*}
    Since $0<b<b^*(a)$, $d_b$ is positive. We define the curves $\rho_{0}(\theta) = \mu(\theta)+d_b$. By Lemma \ref{lemma_distance}, we have that $\mu(\theta)>\rho_{0}(\theta)\geq\frac{1}{a}$ for all $\theta\in\T$. For all $n\geq0$, we define
    \begin{align*}
        \rho_{n+1}(\theta) &:= h_a(h_{a}(\rho_n(\theta-2\omega)) - b g(\theta-2\omega)) - b g(\theta-\omega).
    \end{align*}
    Since $\rho_{0}(\theta)\geq\frac{1}{a}$ and $a d_b>0$, we have that
    \begin{align*}
        \rho_{1}(\theta) &= h_a(a(\rho_0(\theta-2\omega)) - b g(\theta-2\omega)) - b g(\theta-\omega)\\
        &= h_a(a\mu(\theta-2\omega) + a d_b - b g(\theta-2\omega)) - b g(\theta-\omega)\\
        &= h_a(\mu(\theta-\omega) +a d_b) - b g(\theta-\omega)\\
        &= a\mu(\theta-\omega) + a^2 d_b - b g(\theta-\omega)\\
        &= \mu(\theta) + a^2 d_b = \rho_0(\theta) + (a^2-1) d_b.
    \end{align*}
    By induction assume $\rho_n(\theta)\geq \rho_{n+1}(\theta)$. Hence,
    \begin{align*}
        h_a(\rho_{n}(\theta-2\omega))- b g(\theta-2\omega) \leq h_a(\rho_{n+1}(\theta-2\omega))- b g(\theta-2\omega).
    \end{align*}
    Then we have that
    \begin{align*}
        \rho_{n+1}(\theta) &= h_a(a(\rho_{n}(\theta-2\omega)) - b g(\theta-2\omega)) - b g(\theta-\omega)\\
        &\geq h_a(a(\rho_{n+1}(\theta-2\omega)) - b g(\theta-2\omega)) - b g(\theta-\omega) = \rho_{n}(\theta).
    \end{align*}
    The sequence $\{\rho_{n}\}$ is decreasing and $\rho_{n}(\theta)\geq\varphi_{n}(\theta)$, for all $\theta\in\T$. 
    Since $\{\rho_{n}\}$ is decreasing and bounded, it is convergent to a lower semicontinuous curve $\rho_\infty$.
    To conclude with the proof, we show that $\rho_\infty=\varphi_\infty$. Let $\theta_0\in\T$. By induction, we can see that if $\rho_{i}(\theta_0+i\omega))>\frac{1}{a}$, for $i=2(n-1),\dots,2,0$, then
    \begin{align*}
        \rho_{2n}(\theta_0+n\omega) = \mu(\theta_0+2n\omega) + a^{2n}  d_b.
    \end{align*}
    Since $a^{2n}  d_b<0$ is unbounded, there exists an $n_0$ such that
    \begin{align*}
        \mu_{n_0+1}(\theta_0+(n_0+1)\omega) = 1 - b g(\theta) = \varphi_0(\theta_0+(n_0+1)\omega).
    \end{align*}
    Hence, by the continuity of the functions $\mu_n$, there exists an open interval $I_0\subset\T$ where, for all $\theta\in I_0$, we have
    \begin{align*}
        \rho_{n_0+1}(\theta_0+(n_0+1)\omega) = \varphi_0(\theta_0+(n_0+1)\omega).
    \end{align*}
    Therefore, we conclude that $\mu_\infty=\varphi_\infty$. For the rest of the systems, the proof use similar arguments. We use the distance between the first curve of the first sequence, in Lemma \ref{montonicity_lemma}, to define some curve near the curve that comes from the invariant equation. Then, we use the unboundedness of a particular orbit to show that some iteration of the curve is equal to the first element of the first sequence (for each system) in Lemma \ref{montonicity_lemma}. 
\end{proof}
\begin{remark}
    Let $a$ satisfy Assumption \eqref{assumptions_a} and $0<b<b^*(a)$. The following arguments show there are no more continuous invariant curves.
    \begin{itemize}
        \item For the system \eqref{dyn_sys}, any other invariant curve $\rho$ should be equal to $\tilde\varphi_1$ or $\tilde\gamma_1$ in an open interval. The image of $\rho$ cannot be completely contained in $[-\frac{\pi}{2a},\frac{\pi}{2a}]$.
        \item For the system \eqref{dyn_sys_2}, any other invariant curve $\rho$ should be equal to $\tilde\varphi_2$ in an open interval. The image of $\rho$ cannot be completely contained in $(-\infty,- \delta)$ or in $( \delta,\infty)$.
        \item For the system \eqref{dyn_sys_3}, any other invariant curve $\rho$ should be equal to $\tilde\varphi_3$ in an open interval. The image of $\rho$ cannot be completely contained in $(-\frac{1}{a},\infty)$.
        \item For the system \eqref{dyn_sys_4}, any other invariant curve $\rho$ should be equal to $\tilde\varphi_4$ in an open interval. The image of $\rho$ cannot be completely contained in $(\frac{1}{a},\infty)$.
    \end{itemize}
\end{remark}
By the argument in the proof of Proposition \ref{prop_onlyonecurve}, if $0<b<b^*(a)$, we have two sequences of continuous functions converging to our semicontinuous curve $\varphi_\infty$. One from below and the other one from above. This implies that $\varphi_\infty$ is a continuous curve for each system. 

\subsection{After the nonsmooth bifurcation}\label{sect_after}
First we assume that $b>b^*(a)$ and $a$ satisfies \eqref{assumptions_a}. We prove the regularity of the curves and we determine the number of continuous invariant curves has each system. We proceed similarly to in Section \ref{three_curves_secion}.
\begin{proof}[Proof of Proposition \ref{curve_between}]
    Consider the system \eqref{dyn_sys_4}. 
    Let $\theta_0\in\T$ be the minimum of $\varphi_0-\mu$. Thus, by Lemma \ref{lemma_distance},
    \begin{align*}
        \varphi_0(\theta_0)-\mu(\theta_0) = a\left(\frac{1}{a}-\mu(\theta_0-\omega)\right)<0.
    \end{align*}
    Let $\delta=\frac{1}{a}-\mu(\theta_0-\omega)$. We define $\rho_0(\theta) = \mu(\theta)+\delta$. By definition, $\rho_0(\theta)\geq\frac{1}{a}$ for all $\theta\in\T$. We define the following sequence, for $n\geq0$,
    \begin{align*}
        \rho_{n+1}(\theta) &:= h_a(h_{a}(\rho_n(\theta-2\omega)) - b g(\theta-2\omega)) - b g(\theta-\omega).
    \end{align*}
    We want to show that there is an interval $I_0$ where $\rho_n(\theta)=\varphi_0(\theta)$, for all $\theta\in I_0$ and for some $n$.
    If we have that $\mu(\theta_0-\omega) + a\delta < \frac{1}{a}$, then we can compute that $\rho_{1}(\theta_0) = \varphi_0(\theta_0)$. Hence, by continuity, there is an interval that contains $\theta_0$ for which $\rho_{1}(\theta) = \varphi_0(\theta)$. Otherwise, we have that 
    \begin{align*}
        \rho_{1}(\theta_0) &= h_a(a(\rho_0(\theta_0-2\omega)) - b g(\theta_0-2\omega)) - b g(\theta_0-\omega)\\
        &= h_a(a\mu(\theta_0-2\omega) + a\delta - b g(\theta_0-2\omega)) - b g(\theta_0-\omega)\\
        &= h_a(\mu(\theta_0-\omega) +a\delta) - b g(\theta_0-\omega)\\
        &= a\mu(\theta_0-\omega) + a^2\delta - b g(\theta_0-\omega)\\
        &= \mu(\theta_0) + a^2\delta.
    \end{align*}
    Assume by induction that $\rho_n(\theta_0)=\mu(\theta_0)+a^{2n}\delta$ and that $\mu(\theta_0-\omega) + a^{2n+1}\delta \geq \frac{1}{a}$. Then, 
    \begin{align}
        \rho_{n+1}(\theta_0) &= h_a(a(\rho_n(\theta_0-2\omega)) - b g(\theta_0-2\omega)) - b g(\theta_0-\omega) \nonumber\\
        &= h_a(a\mu(\theta_0-2\omega) + a^{2n+1}\delta - b g(\theta_0-2\omega)) - b g(\theta_0-\omega) \nonumber\\
        &= h_a(\mu(\theta_0-\omega) +a^{2n+1}\delta) - b g(\theta_0-\omega) \nonumber\\
        &= a\mu(\theta_0-\omega) + a^{2n+2}\delta - b g(\theta_0-\omega) \nonumber\\
        &= \mu(\theta_0) + a^{2(n+1)}\delta. \label{unbounded_argument}
    \end{align}
    Since $a<-1$ and $\delta>0$, we have that the term $a^{2(n+1)}\delta>0$ is unbounded.
    Therefore, there exists an $n$ such that $\mu(\theta_0-\omega) + a^{2n+1}\delta < \frac{1}{a}$. Thus, $\rho_{n+1}(\theta)=\varphi_0(\theta)$ in some interval $I_0$. Hence, there is a unique continuous invariant curve. For the systems \eqref{dyn_sys} and \eqref{dyn_sys_2}, the argument is very similar to the one showed here. For the system \eqref{dyn_sys_3}, a very similar argument shows that we have two sequences of continuous curves, one monotonically increasing and the other monotonically decreasing that converge to the same invariant curve, let's say $\rho$. The image of $\rho$ cannot be entirely contained in $\{x>-\frac{1}{a}\}$. Otherwise, as it is invariant, $\rho=\mu$. Define $\delta=\mu(\theta_0-\omega)+\frac{1}{a}$ and $\rho_0(\theta)=\mu(\theta)+\delta$. Hence, using a similar argument as in \eqref{unbounded_argument}, it can be proved that $\rho_n(\theta_o+n\omega)-\mu(\theta_o+n\omega)=a^n\delta$. Therefore, there is no other continuous invariant curve above $\mu$, because the distance between $\rho_n$ and $\mu$ is unbounded. 
\end{proof}

Moreover, once we know that for $0<b<b^*(a)$ and $b>b^*(a)$ the curve $\varphi_\infty$ is continuous, we can prove even more regularity of this curve.
\begin{prop}\label{regularity_prop}
    Let $a$ satisfy Assumption \eqref{assumptions_a} and $0<b<b^*(a)$. If $\varphi_\infty$ is continuous, then $\varphi_\infty$ is piecewise $C^{1+\tau}$ with $\tau>0$.
\end{prop}
\begin{proof}
    Consider the system \eqref{dyn_sys_4}. The proof is very similar to the proof of Theorem \ref{finite_num_iter}, but taking advantage that now $\varphi_\infty$ is continuous. Recall that the set
    \begin{align*}
        B=\left\{\theta\in\T\mid a\varphi_\infty(\theta)-b g(\theta)<\frac{1}{a}\right\}
    \end{align*}
    has positive measure. As in Theorem \ref{finite_num_iter},
    \begin{align*}
        C=\bigcup_{n\geq0}T^{-n}(B)
    \end{align*}
    is a set of total Lebesgue measure. Let $\theta_0\in C$, and consider $(\theta_0,\varphi_0(\theta_0))\in\T\times\R$. Then there exists an $n_0>0$ such that $\theta_0+n_0\omega\in B$. Therefore, we have that 
    \begin{align}\label{eq_outside}
        \frac{1}{a}>a\varphi_\infty(\theta_0+n_0\omega)-b g(\theta_0+n_0\omega)\geq a\varphi_{n_0}(\theta_0+n_0\omega)-b g(\theta_0+n_0\omega).
    \end{align}
    Hence, we have that $\varphi_\infty(\theta_0+(n_0+1)\omega) = \varphi_{n_0+1}(\theta_0+(n_0+1)\omega) = \mathcal{F}_4^{2(n_0+1)}(\varphi_0)(\theta_0+3(n_0+1)\omega)$.
    Since $\varphi_\infty$ is continuous, we have that there exists an open interval $I_0$ containing $\theta_0$ such that equation \eqref{eq_outside} is verified, for all $\theta\in I_0$. Hence, $\varphi_\infty(\theta+(n_0+1)\omega) = \varphi_{n_0+1}(\theta+(n_0+1)\omega)$ for all $\theta\in I_0$. Therefore, if we iterate, we obtain that $\varphi_\infty(\theta+(n_0+m)\omega) = \varphi_{n_0+m}(\theta+(n_0+m)\omega)$ for all $\theta\in I_0$.
    Now, since $\T$ is compact, there exists $m_0>0$ such that $\T\subset\cup_{m=0}^{m_0} (I_0+m\omega)$. Moreover, for all $m=0,\dots,m_0$, we have that $\varphi_\infty(\theta) = \varphi_{n_0+m}(\theta)$ for all $\theta\in I_0+m\omega$.
    
    With similar arguments, we can prove the result for the other systems. As in the proof of Theorem \ref{finite_num_iter}, the main changes come from the definition of the set $B$.
\end{proof}
As a consequence of Proposition \ref{regularity_prop}, we get that the curve $\varphi_\infty$ is piecewise as regular as $g$, which proves Theorem \ref{regularity_thm}.

\subsection{The fractalization mechanism}\label{fractalization_section}
In this section, we prove that if $a$ satisfies \eqref{assumptions_a}, then the attracting invariant or two-periodic curve fractalizes, when $b$ approaches $b^*(a)$ from below. For the system \eqref{dyn_sys_4}, the two-periodic curve $\varphi_\infty$ is precisely the curve whose fractalization we intend to establish.
First, we give a proper definition of the fractalization mechanism in our setting. We follow the idea given in \cite{jorba2008mechanism}, where the authors remark that the Hausdorff dimension does not detect that a smooth curve is becoming "fractal". This is owing to the fact that, as long as the invariant object is a curve, it takes the value 1. In the paper mentioned, they use that, for $b\to b^*(a)$, the $C^0$ norm of the curve keeps bounded while the $C^1$ norm of the curve is unbounded. In our case, we cannot compute the derivative because $h_{i}$ is not $C^1$. We extend the definition of the fractalization process in terms of the Lipschitz constants of the invariant curves. We denote by $\text{Lip}(\T,\R)$ the Banach space of Lipschitz functions from $\T$ to $\R$ endowed with the norm $\norm{\cdot}_L:=\norm{\cdot}_\infty+L(\cdot)$, where $L:\text{Lip}(\T,\R)\to\R$ is the operator that associates a Lipschitz function $f$ to its smallest Lipschitz constant $L(f)$. See \cite{cobzacs2019lipschitz} for further information about the Banach space of Lipschitz functions.

\begin{definition}
    Let $f_b$ be an invariant curve of the systems \eqref{dyn_sys}, \eqref{dyn_sys_2}, \eqref{dyn_sys_3}, \eqref{dyn_sys_4}. We say that a family of invariant curves $\{f_b\}_b$ parametrized by $b$ is \textit{fractalizing for a parameter $b^*$} if the Lipschitz constants associated to each curve of the family $\{f_b|_I\}_b$ becomes unbounded as $b\to b^*$, for any interval $I\subset\T$, while the family of curves $\{f_b\}_b$ remains bounded, i.e., if
    \begin{align*}
        \limsup_{b\to b^*} \frac{L_{f_b|_I}}{\norm{f_b}}_\infty = +\infty.
    \end{align*}
\end{definition}

In certain cases, when the perturbation function $g$ is positive, we can prove that the family of invariant curves is monotone with respect the parameter $b$.
\begin{prop}\label{monotonic_sequence_limit}
    Let $a$ satisfy Assumption \eqref{assumptions_a}. If $g(\theta)\geq0$ for all $\theta\in\T$, then for $0<b_0\leq b_1\leq b^*(a)$ we have the following situation:
    \begin{itemize}
        \item For the systems \eqref{dyn_sys} and \eqref{dyn_sys_4}, we have that $\varphi_\infty(b_1,\theta)\leq\varphi_\infty(b_0,\theta)$ for all $\theta\in\T$.
        \item For the systems \eqref{dyn_sys_2} and \eqref{dyn_sys_3}, we have that $\varphi_\infty(b_0,\theta)\leq\varphi_\infty(b_1,\theta)$ for all $\theta\in\T$.
    \end{itemize}
\end{prop}
\begin{proof}
    Consider the system \eqref{dyn_sys_4}. We have that the corresponding sequence of functions, $\{\varphi_{n}(b_0,\theta)\}_n$ and $\{\varphi_{n}(b_1,\theta)\}_n$, converge each to an invariant curve. For $n=0$, we have
    \begin{align*}
        \varphi_{0}(b_0,\theta) &= 1 - b_0 g(\theta-\omega),\\
        \varphi_{0}(b_1,\theta) &= 1 - b_1 g(\theta-\omega).
    \end{align*}
    For $n>0$, we have that
    \begin{align*}
        \varphi_{n}(b_0,\theta) &= h_4(h_{4}(\varphi_{n-1}(b_0,\theta-2\omega)) - b_0 g(\theta-2\omega)) - b_0 g(\theta-\omega),\\
        \varphi_{n}(b_1,\theta) &= h_4(h_{4}(\varphi_{n-1}(b_1,\theta-\omega)) - b_1 g(\theta-2\omega)) - b_1 g(\theta-\omega).
    \end{align*}
    Since $b_0<b_1$ and $g(\theta)\geq0$ for all $\theta\in\T$, we have that $\varphi_{0}(b_1,\theta) \leq \varphi_{0}(b_0,\theta)$. Therefore, using that $\mathcal{F}_{4}$ preserves order, we have that $\varphi_{n}(b_1,\theta) \leq \varphi_{n}(b_0,\theta)$, for all $n\geq0$ and for all $\theta\in\T$. Taking the limit over $n$ in both sides of the last inequality, we obtain that $\varphi_\infty(b_1,\theta)\leq\varphi_\infty(b_0,\theta)$, for all $\theta\in\T$. For the systems \eqref{dyn_sys}, \eqref{dyn_sys_2} and \eqref{dyn_sys_3}, the proof is similar to the one given.
\end{proof}

In our case, the curves are always bounded. Hence, in order to prove that the family $\{\varphi_\infty(b)\}_b$ of invariant curves fractalize as $b\to b^*(a)$, it is enough to show that the Lipschitz constants of the family of invariant curves becomes unbounded, while $b$ approaches to $b^*(a)$. If we impose $g$ to be positive, we are able to show the fractalization phenomenon.
\begin{proof}[Proof of Theorem \ref{weak_frac_thm}]
    Consider the system \eqref{dyn_sys_4}. For $0<b<b^*(a)$, since $g(\theta)\geq0$ for all $\theta\in\T$, by Proposition \ref{monotonic_sequence_limit}, the family $\{\varphi_\infty(b)\}_b$ is monotone. Therefore, since $\{\varphi_\infty(b)\}_b$ is monotone and bounded, for each $\theta\in\T$
    \begin{align*}
        \lim_{b\to b^*(a)} \varphi_\infty(b,\theta)
    \end{align*}
    exists. For every $\theta\in\T$ and $0<b<b^*(a)$ we have
    \begin{align*}
        \varphi_0(b,\theta)\geq \varphi_\infty(b,\theta) \geq \lim_{b\to b^*(a)} \varphi_\infty(b,\theta)\geq \varphi_\infty(b^*(a),\theta)\geq \mu(\theta).
    \end{align*}
    For every $\theta\in\T$ and $0<b<b^*(a)$ we know that
    \begin{align*}
        \varphi_\infty(b,\theta+\omega) = h_4\circ\varphi_\infty(b,\theta) - b g(\theta).
    \end{align*}
    Hence, taking $b\to b^*(a)$ we get that
    \begin{align*}
        \lim_{b\to b^*(a)}\varphi_\infty(b,\theta+\omega) = h_4\left(\lim_{b\to b^*(a)}\varphi_\infty(b,\theta)\right) - b^*(a) g(\theta).
    \end{align*}
    But the only invariant curve in between $\varphi_0(b^*(a))$ and $\varphi_\infty(b^*(a))$ is $\varphi_\infty(b^*(a))$. Therefore, we conclude that
    \begin{align*}
        \lim_{b\to b^*(a)} \varphi_\infty(b,\theta) = \varphi_\infty(b^*(a),\theta).
    \end{align*}
    Assume by contradiction that
    \begin{align*}
            \limsup_{b\to b^*} L(\varphi_\infty(b)|_I) < +\infty.
    \end{align*}
    Since $\omega$ is irrational and $\T$ is compact, we have that any $\theta\in\T$ can be obtained from a value in $I$ by adding a bounded multiple of $\omega$. Therefore, there exist $K_1$ and $K_2$ depending on $\omega,a,b$ and $I$ such that $L(\varphi_\infty(b,\theta))\leq K_1 L(\varphi_\infty(b,\theta)|_I) + K_2$. This implies that 
    \begin{align*}
            L:=\limsup_{b\to b^*} L(\varphi_\infty(b)) < +\infty.
    \end{align*}
    Then, for $0<b\leq b^*(a)$, the family $\{\varphi_\infty(b,\theta)\}_{b}$ is uniformly bounded and has a uniformly Lipschitz constant $L$. By the Arcelà-Ascoli, Theorem we have that the family $\{\varphi_\infty(b,\theta)\}_{b}$ has a subsequence that converges uniformly to a Lipschitz function with Lipschitz constant $L$. Therefore, the curve $\varphi_\infty(b^*(a),\theta)$ is a Lipschitz function with Lipschitz constant $L$. In contradiction with the result that shows that $\varphi_\infty(b^*(a),\theta)$ is a noncontinuous curve (Theorem \ref{big_thm}).
\end{proof}

A consequence of Arcelà-Ascoli Theorem is that a sequence of Lipschitz curves that converge pointwise to a noncontinuous curve cannot have a uniform Lipschitz constant. Moreover, assume we have a family of Lipschitz curves $\{\eta_n\}_n$ defined over $\T$ that converge pointwise to a noncontinuous invariant curve $\eta_\infty$. Then for any interval $I\subset\T$,
    \begin{align*}
        \liminf_{n} L(\eta_n|_I) = +\infty.
    \end{align*}

Note that, for fixed $b=b^*(a)$, this is the case for the monotone sequence of functions defined in Lemma \ref{montonicity_lemma}. Therefore, we have this concluding result that characterizes our nonsmooth invariant curve.
\begin{proof}[Proof of Theorem \ref{mon_characterization}]
    If $f_n(\theta):=h\circ f_{n-1}(\theta-\omega)- b g(\theta-\omega)$ has pointwise limit, then take the limit in both sides to get that $f_\infty(\theta):=h\circ f_{\infty}(\theta-\omega)- b g(\theta-\omega)$.
    Assume that $f_\infty\leq f_0$, the case $f_\infty\geq f_0$ follows by similar arguments for $\{f_m^{\uparrow}\}_m$. Since each $f_n$ is continuous, $f_m^{\downarrow}(\theta)=\inf_{k\leq m}f_k(\theta)$ is continuous and $f_\infty(\theta)\leq f_{m}^{\downarrow}(\theta)\leq f_m(\theta)$, for all $m\geq0$ and for all $\theta\in\T$. Since $f_n$ converges pointwise to $f_\infty$, $f_m^{\downarrow}$ converges pointwise to $f_\infty$. For all $\theta\in\T$ and all $m\geq0$, we have that $f_\infty(\theta)\leq f_{m+1}^{\downarrow}(\theta)\leq f_m^{\downarrow}(\theta)$. Hence, the sequence $\{f_m^{\downarrow}\}_m$ is monotone. By Dini's Theorem, we have that $f_\infty$ is continuous if and only if $\{f_m^{\downarrow}\}_m$ is uniformly convergent.
\end{proof}
\begin{proof}[Proof of Corollary \ref{char_corollary}]
    Since $\T$ is compact and $\{\varphi_n\}_n$ is monotone, Dini's Theorem implies that $\varphi_\infty$ is noncontinuous if and only if $\{\varphi_n\}_n$ converges uniformly.
\end{proof}

\begin{remark}
    Moreover, under the same hypothesis of Theorem \ref{mon_characterization}, if additionally we require that $g$, $h$ and $f_0$ are all Lipschitz continuous, then from the Arcelà-Ascoli Theorem we obtain that:
    \begin{itemize}
        \item If $\{f_m^{\downarrow}\}_m$ has uniform Lipschitz constant, then the convergence of $\{f_m^{\downarrow}\}_m$ is uniform and $f_\infty$ is Lipschitz continuous.
        \item If $\{f_m^{\uparrow}\}_m$ has uniform Lipschitz constant, then the convergence of $\{f_m^{\uparrow}\}_m$ is uniform and $f_\infty$ is Lipschitz continuous.
    \end{itemize}
\end{remark}
\begin{proof}
    The invariance of $f_\infty$ follows from the same argument in the proof of Theorem \ref{mon_characterization}. Assume that $f_\infty\leq f_0$, the case $f_\infty\geq f_0$ follow by similar arguments for $\{f_m^{\uparrow}\}_m$. For $n\geq0$, we have that $\{f_n\}_n$ is a sequence of Lipschitz functions.
    Note that, for all $\theta\in\T$, we have that $f_0^{\downarrow}(\theta)=f_0(\theta)$ and
    \begin{align*}
        f_m^{\downarrow}(\theta)=\inf\{f_{m-1}^{\downarrow}(\theta), f_m(\theta)\}=\frac{f_{m-1}^{\downarrow}(\theta)+f_m(\theta)+\abs*{f_{m-1}^{\downarrow}(\theta)-f_m(\theta)}}{2}.
    \end{align*}
    Therefore, by induction we get that $\{f_m^{\downarrow}\}_m$ is a sequence of Lipschitz curves. If $\{f_m^{\downarrow}\}_m$ or $\{f_m^{\uparrow}\}_m$ have a uniformly Lipschitz constant, then by Arcelà-Ascoli Theorem $f_\infty$ is a Lipschitz curve.
\end{proof}

\subsection{Uniform contraction case}\label{one_inv_curv}
In this section, we consider that $\abs{a}<1$. In this case, the piecewise-linear quasiperiodically forced dynamical system has a uniform contraction. Therefore, we can prove that we have a unique Lipschitz invariant curve.

Recall the definition of the seminorm 
\begin{align*}
    L(f)=\sup\{\abs{f(x)-f(y)}/\abs{x-y}\mid x,y\in\T,x\not=y\}.
\end{align*}
Recall that an invariant curve of $F_i$ is a fixed point of the map $\mathcal{F}_i$.
In the four piecewise-linear systems, the map $\mathcal{F}_i$ is a Lipschitz map with Lipschitz constant $\abs{a}$. Actually,
\begin{align*}
    \abs{\mathcal{F}_i(\varphi)(\theta)-\mathcal{F}_i(\psi)(\theta)}\leq \abs{a}\abs{\varphi(\theta-\omega)-\psi(\theta-\omega)}.
\end{align*}
Therefore, we obtain that
\begin{align*}
    \norm{\mathcal{F}_i(\varphi)-\mathcal{F}_i(\psi)}_L &= \norm{\mathcal{F}_i(\varphi)-\mathcal{F}_i(\psi)}_\infty + L(\mathcal{F}_i(\varphi)-\mathcal{F}_i(\psi))\\     
    &\leq \abs{a}\norm{\varphi-\psi}_\infty + \abs{a} L(\varphi-\psi) = \abs{a}\norm{\varphi-\psi}_L.
\end{align*}
As a consequence of the Banach fixed point Theorem, we obtain Theorem \ref{thm_uniform_contraction_case}.

\section{Conclusions}
We conclude the paper with a summary of the main results and a discussion of possible extensions and open problems.
We proved the existence of nonsmooth bifurcations in four piecewise-linear quasiperiodically forced maps. For the bifurcation parameters, we established that the closure of the attracting invariant set is a region of two-dimensional Lebesgue positive measure. Additionally, we demonstrated that an invariant (or two-periodic) curve undergoes a fractalization phenomenon as the parameters approach the curve of bifurcation. Finally, under suitable assumptions, we characterized the noncontinuity of the attracting curve in terms of the uniform convergence of a sequence of continuous curves.

We have extended the results of \cite{jorba2024nonsmooth}. In particular, our second result concerning the closure of the attracting invariant set provides a positive answer to a question in \cite{keller1996sna}. Both the present work and \cite{jorba2024nonsmooth} contribute to clarify the mechanisms that lead to the creation of strange nonchaotic attractors in piecewise-linear quasiperiodically forced maps.

Regarding future work and the open problems, we mention two natural extensions. First, to study higher-dimensional examples for which similar results can be obtained. Second, it would be of interest to adapt these techniques to smooth models, in order to find more mechanisms for detecting nonsmooth bifurcations.

\section*{Acknowledgements}
\thanks{
This work has been supported by the Spanish grant PID2021-125535NB-I00 funded by MICIU/AEI, Spain/10.13039/501100011033 and by ERDF/EU, Spain. 
R.M.V. and J.C.T. also acknowledge the Catalan, Spain grant 2021-SGR-01072.
\vspace{0.25cm}

With profound sadness and respect, this publication is dedicated to the memory of our collaborator, friend, and esteemed Prof. Àngel Jorba. He was part of the early development of this work. His passing, prior to its completion, represents a deeply felt absence. We honour his legacy and his significant contributions.
\vspace{0.25cm}

This version of the article has been accepted for publication, after peer review (when applicable) but is not the Version of Record and does not reflect post-acceptance improvements, or any corrections. The Version of Record is available online at: https://doi.org/10.1007/s12346-025-01438-0.
}

\section*{ORCID}
Rafael Martinez-Vergara ORCID: \href{https://orcid.org/0009-0006-8350-589X}{0009-0006-8350-589X}.\\
Joan Carles Tatjer ORCID: \href{https://orcid.org/0000-0001-8309-3940}{0000-0001-8309-3940}.

\bibliographystyle{alpha}
\bibliography{biblio}

\end{document}